\documentclass[12pt]{article}

\usepackage{amsthm}
\usepackage{amsmath}
\usepackage{amssymb}
\usepackage{authblk}
\usepackage[toc,page]{appendix}
\usepackage{graphicx}
\usepackage{booktabs}
\usepackage[round]{natbib}
\usepackage[super]{nth}
\usepackage{bbm}
\usepackage[misc]{ifsym}
\usepackage{textcomp}
\usepackage{appendix}
\usepackage{enumitem}
\usepackage{pgf,tikz,pgfplots}
\usepackage{mathrsfs}
\usetikzlibrary{arrows}
\usetikzlibrary{shapes,backgrounds}
\usepackage{color}
\usepackage{framed}
\usepackage{comment}
\definecolor{shadecolor}{gray}{0.9}

\theoremstyle{plain}  
\newtheorem{theorem}{Theorem}[section] 
\newtheorem{lemma}[theorem]{Lemma} 
\newtheorem{proposition}[theorem]{Proposition} 
\newtheorem{corollary}[theorem]{Corollary}

\theoremstyle{definition} 
\newtheorem{definition}[theorem]{Definition}

\newtheorem{example}[theorem]{Example}

\theoremstyle{remark} 
\newtheorem{remark}[theorem]{Remark}

\newcommand{\diff}{\,\mathrm{d}}
\newcommand{\E}{\mathbb{E}}

\newcommand{\R}{\mathbb{R}}

\newcommand{\one}{\mathbbm{1}}

\newcommand{\argmin}{\operatorname{arg\,min}}

\begin{document}

\title{Optimal solutions to the isotonic regression problem}
\author{Alexander I.~Jordan}
\author{Anja M\"uhlemann}
\author{Johanna F.~Ziegel}
\affil{University of Bern}
\maketitle
\begin{abstract}
	In general, the solution to a regression problem is the minimizer of a given loss criterion, and depends on the specified loss function.
	The nonparametric isotonic regression problem is special, in that optimal solutions can be found by solely specifying a functional.
	These solutions will then be minimizers under all loss functions simultaneously as long as the loss functions have the requested functional as the Bayes act.
	For the functional, the only requirement is that it can be defined via an identification function, with examples including the expectation, quantile, and expectile functionals.	\\
	Generalizing classical results, we characterize the optimal solutions to the isotonic regression problem for such functionals, and extend the results from the case of totally ordered explanatory variables to partial orders.
	For total orders, we show that any solution resulting from the pool-adjacent-violators algorithm is optimal.
	It is noteworthy, that simultaneous optimality is unattainable in the unimodal regression problem, despite its close connection.\\
	
\end{abstract}

\section{Introduction}\label{sec:intro}
Suppose that we have pairs of observations $(z_1, y_1)$, $\dots, (z_n, y_n)$ where we assume that $y_i$, $i=1, \dots, n$ are real-valued.
The aim of isotonic regression is to fit an increasing function $\hat{g}\colon \{z_1, \dots, z_n\} \to \R$ to these observations.
The covariates $z_1,\dots, z_n$ can take values in any set as long as it is equipped with a partial order which we denote by $\preceq$.
Then, a function $g\colon \{z_1, \dots, z_n\} \to \R$ is \emph{increasing} if $z_i \preceq z_j$ implies that $g(z_i) \le g(z_j)$.

As it is common in regression analysis, we aim to find an estimate $\hat{g}$ that minimizes the expected loss for some loss function $L\colon \R \times \R \to [0,\infty)$.
If the function $\hat{g}$ is interpreted as an estimator of the conditional expectation of a random variable $Y$ given $Z$, then a natural choice for $L$ is the squared error loss $L(x,y) = (x-y)^2$.
For $i \le j$, let $\E_{i:j}$ denote the expectation with respect to the empirical distribution of $(z_i,y_i), \dots, (z_j,y_j)$.
Assuming that $z_1 < z_2 < \dots < z_n$, the minimizer of the quadratic loss criterion
\begin{equation} \label{eq:criterion}
\E_{1:n} (g(Z) - Y)^2
\end{equation}
over all increasing functions $g$ is given by
\begin{equation} \label{eq:solution}
\hat{g}(z_\ell) = \min_{j \ge \ell}\max_{i \le j} \E_{i:j} Y = \max_{i\le \ell}\min_{j \ge i}\E_{i:j} Y, \quad \ell =1, \dots, n,
\end{equation}
see \citet[eq.~(1.9)--(1.13)]{Barlow1972}.
The solution $\hat{g}$ can be computed efficiently using the so-called pool-adjacent-violators (PAV) algorithm.
These results were developed in the 1950s by several parties independently; see \citet{Ayer1955}, \citet{Bartholomew1959a}, \citet{Bartholomew1959b}, \citet{Brunk1955}, \citet{vanEeden1958}, \citet{Miles1959}.

It turns out that the solution given at \eqref{eq:solution} is also the unique minimizer of the Bregman loss criterion
\begin{equation}\label{eq:criterion2}
\E_{1:n} L (g(Z), Y),
\end{equation}
where the squared error loss in \eqref{eq:criterion} has been replaced by a Bregman loss function $L=L_\phi$ \citep[Theorem 1.10]{Barlow1972}.
That is, 
\begin{equation*}
L_\phi(x,y) = \phi(y) - \phi(x) - \phi'(x)(y-x),
\end{equation*}
where $\phi$ is a convex function with subgradient $\phi'$.
\citet{Savage1971} found that the Bregman class comprises all loss functions $L$ where the expectation functional minimizes the expected loss, i.e.,
\[
\E_P Y = \argmin_x \E_P L(x, Y),
\]
where $Y$ is a random variable with distribution $P$.
Due to this property, any loss function in the Bregman class is also referred to as a consistent loss function for the expectation functional \citep{Gneiting2011}.

In summary, the increasing regression function at \eqref{eq:solution} is simultaneously optimal with respect to all consistent loss functions for the expectation.
This robustness with respect to the choice of loss function means that the solution to the regression problem is determined by the choice of the expectation as the target functional.
We will see that the same holds for other functionals.
As such, in nonparametric isotonic regression we can replace the task of choosing a loss function with the task of choosing a suitable target functional.

This remarkable result is particularly beneficial in scenarios where a single relevant loss function cannot easily be identified.
For example, institutions such as central banks or weather services provide analyses and forecasts that drive individual decision making in a heterogeneous group of users.
In these circumstances, determining a unifying loss function is hardly trivial.
However, publishing results for the expectation and for various quantile levels is certainly feasible.

The simultaneous-optimality result for nonparametric isotonic regression is in stark contrast to the optimality behavior of parametric models for increasing regression functions.
Suppose that $\{g_\theta : \theta \in \Theta\}$, $\Theta \subseteq \R^d$ is a parametric model of increasing functions $g_\theta$.
Then, the optimal parameters with respect to the Bregman-loss criterion \eqref{eq:criterion2} generally vary (substantially) depending on the chosen loss function \citep{Patton2018}.
Consistency of the loss function merely ensures that the true parameter value of a correctly specified model minimizes the Bregman-loss criterion on the population level.
Interestingly, simultaneous optimality with respect to all consistent loss functions generally also breaks down if one weakens the isotonicity constraint of the regression function to a unimodality constraint; see Section \ref{sec:unimodal}.

In this paper, we generalize the result of \citet[Theorem 1.10]{Barlow1972} in several directions.
First, instead of the expectation functional, we consider general (possibly set-valued) functionals $T$ that are given by an identification function $V(x,y)$ as defined in Definition \ref{def:TviaV}.
Second, in the case of set-valued functionals, we give a complete characterization of all possible solutions for totally ordered covariates.
Third, we demonstrate that a suitably modified version of min-max or max-min solutions as in \eqref{eq:solution} continues to hold for general partial orders on the covariates.

\begin{figure}
	\centering
	\includegraphics{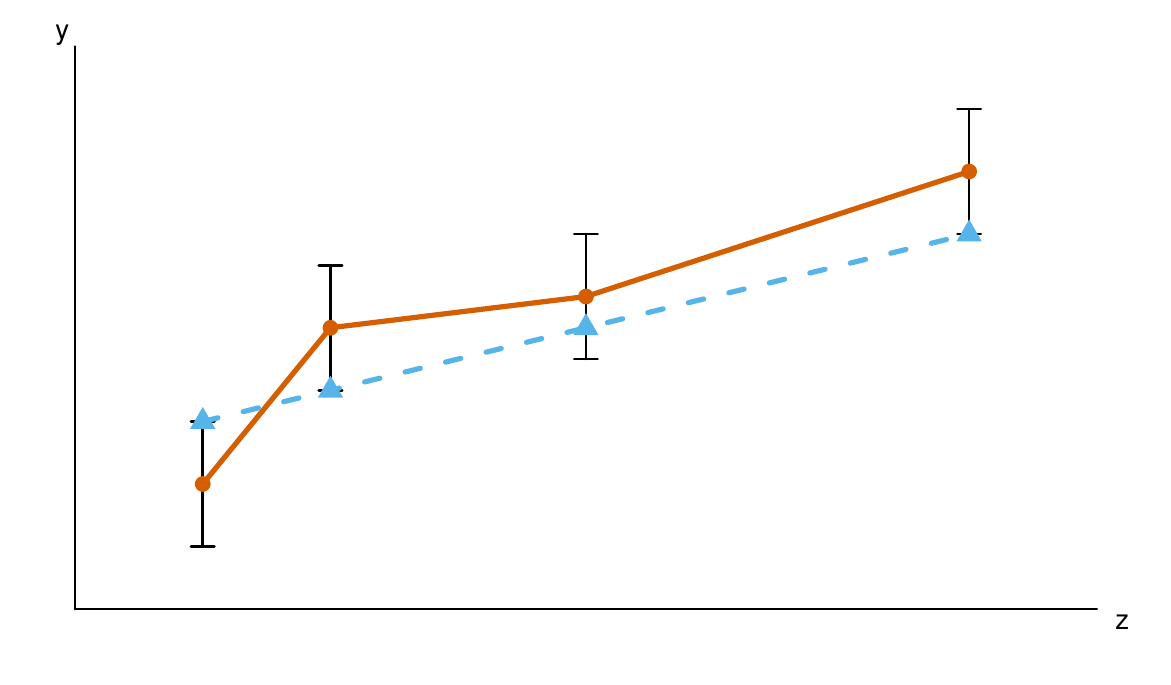}
	\caption{\textbf{Solutions in isotonic quantile regression.}
		Two solutions to the isotonic regression problem are shown for an example with $z = z_1, \dots, z_4$.
		The red curve is predetermined to pass through the midpoint of the functional intervals, whereas the blue curve illustrates the smoothest solution with minimal slope \label{fig:motivation}}
\end{figure}

An identification function is an increasing function that weighs negative values in the case of  underestimation against positive values in the case of overestimation, with an optimal expected value of zero.
The corresponding functional $T$ then maps to the optimizing argument (or set of optimizing arguments).
Prime examples of such functionals are (possibly set-valued) quantiles, expectiles \citep{Newey1987}, or ratios of expectations.
Quantiles, including the median, have previously also been treated in \citet{RobertsonWright1973, RobertsonWright1980}, but not in the interpretation as set-valued functionals.
Predefining a global scheme for reducing the median interval to a single point (e.g., some weighted average of lower and upper functional value) inevitably restricts the possible solutions to the isotonic regression problem.
Figure \ref{fig:motivation} illustrates this issue, and shows how a more general interpretation of the functional as set-valued facilitates solutions with secondary optimality criteria such as smoothness and minimal slope.
Expectiles and ratios of expectations, on the other hand, have been fully treated in \citet{RobertsonWright1980}.
These functionals map to single values and satisfy the Cauchy mean value property which is implied by identifiability.

In contrast to previous work, we treat all functionals as set-valued.
In Section \ref{sec:total}, we give explicit solutions for the lower and upper bound of the isotonic regression problem in the context of total orders.
The method of proof for these results is fundamentally different from the approach of \citet[Theorem 1.10]{Barlow1972} or \citet{RobertsonWright1980}, and in contrast to the latter comes with an immediate construction principle for loss functions.
Our method relies on the mixture or Choquet representations of consistent loss functions, introduced by \citet{Ehm2016} for the quantile and expectile functionals.
Given the identification function $V(x,y)$ for the functional $T$, a one-parameter family of elementary loss functions that are consistent for the functional $T$ can be readily defined,
\begin{align*}
S_\eta(x,y) = \left(\one\{\eta \leq x\} - \one\{\eta \leq y\}\right)V(\eta,y),
\end{align*}
where $\eta \in \R$.
For all consistent loss functions $L$ in the class
\begin{equation}\label{eq:calS}
\mathcal{S}
=\left\lbrace \int_\R S_\eta(x,y) \, \mathrm{d}H(\eta): H \text{ is a nonnegative measure on }\R \right\rbrace,
\end{equation}
the optimal isotonic solution to the criterion \eqref{eq:criterion2} is bounded below by a min-max formula and bounded above by a max-min formula as in \eqref{eq:solution} with the expectation replaced by the lower and upper functional values under $T$, respectively.
We show that the min-max or max-min solution is simultaneously optimal with respect to all elementary loss functions for $T$, and hence with respect to the entire class $\mathcal{S}$.
In fact, optimality of an isotonic solution with respect to the criterion \eqref{eq:criterion2} for $L=S_\eta$ for some $\eta \in \R$ corresponds to finding a solution with optimal superlevel set $\{g \ge \eta\}$.
Considering an isotonicity constraint as a constraint on admissible superlevel sets of the regression function relates to the work of \citet{Polonik1998} in the context of density estimation.

If $T$ is a quantile, an expectile, or a ratio of expectations, then $\mathcal{S}$ comprises all consistent loss functions for $T$ subject to standard conditions, and if $V(x, y) = x - y$ is the identification function of the expectation, then the class $\mathcal{S}$ is the class of Bregman loss functions; see \citet{Ehm2016, Gneiting2011}.
We also give results that can be directly translated to a simple algorithm that recovers the full range of optimal solutions from the lower and upper bounds and the full data set.
While the bounds alone do not contain sufficient information, only few additional computations on the entire data set are necessary.
Our method of proof also leads to a transparent proof of the validity of the PAV algorithm; see Section \ref{sec:PAV}.

Recently, \citet{MoschingDumbgen2019} derived a similar result of min-max and max-min formulas as lower and upper bounds for optimal isotonic solutions in the context of set-valued minimizers of convex and coercive loss functions.
\citet{Brummer2013} rediscover the result of \citet{Barlow1972} that the PAV algorithm leads to a simultaneously optimal solution for all proper scoring rules in the context of binary events -- a special class of loss functions that are consistent for the expectation functional.

In Section \ref{sec:partial}, we treat general partial orders on the covariates and demonstrate that a suitably modified version of min-max or max-min solutions continues to hold.
Again, the optimal isotonic fit is simultaneously optimal with respect to all loss functions in $\mathcal{S}$ defined at \eqref{eq:calS}.
With our method of proof this extension is straightforward but for reasons of transparency, we first present the case of a total order in Section \ref{sec:total}.
The results in \citet{RobertsonWright1980} not only hold for a large class of functionals, but also for partial orders on the covariates.
However, the generality of their results is limited by treating potentially set-valued functionals as maps to single values.
To the best of our knowledge, the literature following \citet{RobertsonWright1980} is void of further results that characterize the solutions to the isotonic regression problem, or any investigations into the effect of the choice of loss function among options sharing the same Bayes act.

A comprehensive overview on isotonic regression is given in the monograph \citet{Groeneboom2014}.
Also, \citet{Guntuboyina2018} review risk bounds, asymptotic theory, and algorithms in common nonparametric shape-restricted regression problems in the context of least squares optimization.
Among the most recent developments on algorithms for isotonic regression with partially ordered covariates, \citet{Kyng2015} and \citet{Stout2015} provide fast algorithms for isotone regression under different loss functions using the representation of a partial order as a directed acyclic graph.
Recent advances on asymptotic theory for isotonic regression include \citet{Han2017}, giving rates for least squares isotonic regression on the unit cube of arbitrary dimension, and \citet{Bellec2018}, considering isotonic, unimodal, and convex regression in the context of total orders.
Another recent interest is the regularization of isotonic regression on multiple variables with \citet{Luss2017} proposing a method via range restriction on the solution to the regression problem.

\section{Functionals and consistent loss functions} \label{sec:func}
We start with the definition of a functional via an identification function.
\begin{definition} \label{def:TviaV}
	A function $V\colon \R \times \R \to \R$ is called an \emph{identification function} if $V(\cdot,y)$ is increasing and left-continuous for all $y \in \R$.
	Then, for any finite and nonnegative measure $P$ on $\R$, we define the \emph{functional} $T$ induced by an identification function $V$ as
	\begin{equation*}
	T(P) = [T_P^-, T_P^+] \subseteq [-\infty, +\infty] = \bar{\R},
	\end{equation*}
	where the lower and upper bounds are given by
	\[
	T_P^- = \sup \left\{x : V(x, P) < 0\right\} \quad \text{and} \quad
	T_P^+ = \inf \left\{x : V(x, P) > 0\right\},
	\]
	using the notation $V(x, P) = \int_{-\infty}^{\infty} V(x, y) \diff P(y)$.
\end{definition}
Defining functionals for any finite and nonnegative measure, as opposed to merely probability distributions, is a minor detail that simplifies notation when joining and intersecting data subsets.
Except in the case of the null measure, any finite and nonnegative measure can be replaced with its corresponding probability distribution, without any change to the functional values.

All results are concerned with probability distributions $P$ with finite support, and therefore, the existence of integrals is guaranteed.
The following example and Proposition \ref{prop:consistent} hold for more general types of distributions given that the relevant integrals exist.
We leave these obvious generalizations up to the reader and assume that all probability distributions considered have finite support.

Note that $T_P^-$ can take the value $-\infty$ and $T_P^+$ can take the value $+\infty$.
In the subsequent results, we repeatedly refer to the smallest or largest element of a finite set where one of the elements could be $\pm \infty$.
We still write $\min$ and $\max$ of the set but this quantity could be $\pm \infty$.

\begin{definition}
	A functional $T$ is called a \emph{functional of singleton type} if $T(P)$ is a singleton whenever $P$ is not the null measure.
	Otherwise, $T$ is called a \emph{functional of interval type}.
\end{definition}

Table \ref{tab:functional} summarizes common functionals and their respective identification functions, and Example \ref{ex:functional} explains two options in more detail.

\begin{table}
	\caption{\textbf{Selection of functionals and their respective identification functions.} The parameters satisfy $\alpha, \tau\in (0,1)$, $p > 1$ and $\delta > 0$, and $u:I \to \R$ and $w:I \to (0,\infty)$ are measurable functions on an interval $I \subseteq \R$.
		The functionals ``$\ell_p$ minimizer'' and ``Huber minimizer'' map to the intervals of values minimizing the $\ell_p$ loss and the Huber loss \citep{Huber1964}, respectively} \label{tab:functional}
	\vspace{0.1cm}
	\begin{center}
		\begin{tabular}{lll} 
			\toprule
			Functional & Identification function & Type \\ 
			\midrule
			Median & $V(x,y)= \one\{x>y\}-1/2$ & interval\\
			Mean & $V(x,y) = x-y$ & singleton\\ 
			$\nth{2}$ Moment & $V(x,y)=x-y^2$ & singleton\\
			$\alpha$-Quantile & $V(x,y)= \one\{x>y\}-\alpha$& interval\\ 
			$\tau$-Expectile & $V(x,y)=2 \vert \one\{x>y\}-\tau \vert (x-y)$ & singleton\\
			Ratio $\mathbb{E}_P(u(Y))/\mathbb{E}_P(w(Y))$ & $V(x,y)=xw(y)-u(y)$ & singleton\\
			$\ell_p$ minimizer & $V(x, y) = \operatorname{sign}(x - y) \vert x - y \vert^{p - 1}$ & singleton\\
			Huber minimizer & $V(x,y)= \operatorname{sign}(x-y) \min(\vert x-y \vert, \delta)$& interval\\
			\bottomrule
		\end{tabular}
	\end{center}
\end{table}

\begin{example} \label{ex:functional}
	Let  $\alpha, \tau \in (0,1)$, and let $P$ denote a probability distribution.
	\begin{enumerate}[label=(\alph*)]
		\item
		Consider the identification function $V(x,y)=\one\{x > y\}-\alpha$, then
		$V(x, P) = P(Y < x) - \alpha$, and the interval of all $\alpha$-quantiles of $P$,
		\[
		T(P)
		= [\sup \{x : P(Y < x) < \alpha\}, \inf \{x : P(Y < x) > \alpha\}],
		\]
		is potentially of positive length.
		\item
		The identification function $V(x,y)=2|\one\{x > y\}-\tau|(x-y)$ leads to
		\begin{align*}
		V(x, P) 
		= 2 (1-\tau) \int_{-\infty}^x (x-y) \, \mathrm{d}P(y) + 2 \tau \int_{x}^\infty (x-y) \, \mathrm{d}P(y),
		\end{align*}
		which is strictly increasing and continuous in its first argument.
		Hence, there exists a unique solution in $x$ for the equation $V(x, P) = 0$, and we call that solution the $\tau$-expectile $e_\tau(P)$.
		In particular, for $\tau=\frac{1}{2}$ we obtain $V(x,y)=x-y$ and thus $T(P)=\{ \E_P(Y)\}$.
	\end{enumerate}
\end{example}

In the later proofs, we use three implications of Definition \ref{def:TviaV} repeatedly to establish order relationships between the variable in the first argument of $V$ and the functional of an empirical distribution.
To facilitate reference, we note these statements explicitly.
\begin{corollary} \label{cor:TviaV}
	Let $V$ be an identification function inducing the functional $T$, and $P$ be a finite and nonnegative measure on $\R$. Then,
	\begin{align*}
	&V(\eta, P) = 0 \implies \eta \in T(P),\\
	&V(\eta, P) > 0 \implies \eta > \sup T(P) = T_P^+,\\
	&V(\eta, P) < 0 \implies \eta \le \inf T(P) = T_P^-.
	\end{align*}
\end{corollary}

Lemma \ref{lemma:functional_bounds} shows that a generalized version of the Cauchy mean value property, used to define functionals in \citet{RobertsonWright1980}, holds for any functional we consider in this paper.
This suggests that our results are less general, unless it can be proven that every Cauchy mean value function can be defined in terms of an identification function.
On the other hand, in contrast to \citet{RobertsonWright1980}, we treat set-valued functionals and their boundaries rigorously, and retain a higher level of generality in that regard.
\begin{lemma} \label{lemma:functional_bounds}
	Let $P, Q$ be finite and nonnegative measures on $\R$.
	Then,
	\[
	\min\{T_P^-, T_Q^+ \} \le T_{P + Q}^- \le T_{P + Q}^+ \le \max\{ T_P^-, T_Q^+ \}.
	\]
\end{lemma}
\begin{proof}
	The statement follows from Definition \ref{def:TviaV}.
	The second inequality is trivial.
	For the first inequality, and $x < \min\{T_P^-, T_Q^+\}$, we have $V(x, P) < 0$ and $V(x, Q) \le 0$, hence $V(x, P + Q) < 0$.
	A similar argument applies to the third inequality.
\end{proof}

The definition of a functional in terms of an identification function comes with a straightforward construction principle for large classes of loss functions.
In a nutshell, a continuous oriented identification function defines a functional via its unique root in the first argument, a first-order condition.
By integration, corresponding loss functions inherit the consistency for the functional, i.e., the minimum expected loss is attained by any member in $T(P)$.
The loss functions defined in Proposition \ref{prop:consistent} are the most basic, in the sense that they are a result of integration with respect to the Dirac measure at a given threshold $\eta \in \R$.
A similar result has also been discussed in \citet{Dawid2016} and \citet{Ziegel2016}.

\begin{proposition} \label{prop:consistent}
	Let $V$ be an identification function, $T$ be the induced functional, and $\eta \in \R$.
	Then the elementary loss function $S_\eta \colon \bar{\R} \times \R \to \R$ given by
	\begin{align*}
	S_\eta(x,y) 
	= \left(\one\{\eta \leq x\} - \one\{\eta \leq y\}\right)V(\eta,y)
	\end{align*}
	is \emph{consistent} for $T$ relative to the class $\mathcal{P}$ of probability distributions with finite support.
	That is, 
	\begin{align*}
	\E_P S_\eta(t,Y) \leq \E_P S_\eta(x,Y)
	\end{align*}
	for all $P\in \mathcal{P}$, all $t \in T(P)$ and all $x \in \bar{\R}$.
\end{proposition}
\begin{proof}
	Let
	\begin{align*}
	d(\eta) = \E_P S_\eta(t,Y) - \E_P S_\eta(x,Y)
	= \left(\one\{\eta \leq t\} - \one\{\eta\leq x\}\right) V(\eta,P).
	\end{align*}
	If $V(\eta,P)=0$ then $d(\eta) = 0$.
	If $V(\eta,P)<0$ it follows from Corollary \ref{cor:TviaV} that $\eta \leq t$ and therefore $d(\eta)\leq 0$.
	Similary, if $V(\eta,P)>0$ it follows that $\eta > t$ and therefore $d(\eta) \leq 0$.
\end{proof}

As an immediate consequence of the consistency of elementary loss functions for the functional $T$, we have that all loss functions in the class $\mathcal{S}$ defined at \eqref{eq:calS} are also consistent for the functional $T$.
This result exemplifies an important line of reasoning used multiple times in this paper: A property of $S_\eta$ that holds for all $\eta \in \R$ translates to the class $\mathcal{S}$.
Examples of members of the class $\mathcal{S}$ for the expectation functional, i.e., $V(x, y) = x - y$, are given in Table \ref{tab:loss}.

\begin{table}
	\begin{center}
		\caption{\textbf{Commonly used loss functions that are consistent for the mean functional.} For an interval  $I\subseteq \R$, a Bregman loss is induced by a convex function $\phi:I \to \mathbb{R}$ with subgradient $\phi'$.
			See \citet{Patton2011, Patton2018} for the QLIKE loss and the exponential Bregman loss, respectively} \label{tab:loss}
		\vspace{0.1cm}
		\small
		\begin{tabular}{llll} 
			\toprule
			Name &  Mixing measure & Loss function & Domain \\
			&  $H((\eta_1,\eta_2])=$ & $L(x,y)=$ & \\
			\midrule
			Bregman loss & $\phi'(\eta_2) - \phi'(\eta_1)$ &
			$\phi(y)-\phi(x)-\phi'(x)(y-x)$ & $I$\\
			Squared error & $\eta_2-\eta_1$ &
			$(x-y)^2$ & $\R$\\
			Exponential Bregman & $\exp(\eta_2)-\exp(\eta_1)$ &
			$\exp(y)-\exp(x)-\exp(x)(y-x)$ & $\R$\\
			QLIKE loss & $-1/\eta_2 + 1/\eta_1$ &
			$y/ x-\log\left(y/ x\right)-1$ & $(0,\infty)$\\
			\bottomrule
		\end{tabular}
	\end{center}
\end{table}

The importance of the construction in Proposition \ref{prop:consistent} lies in the postponing of integration, or, in other words, applying Fubini in a double integration (with respect to $P$ and to $H$), and then showing the property of consistency for the integrand $S_\eta$ for each $\eta$ rather than for the original loss function which is the integral of $S_\eta$ with respect to $\diff H(\eta)$.

\section{Results for total orders} \label{sec:total}

\subsection{Min-max and max-min solutions} \label{sec:minmax}
Suppose that we have observations $(z_1, y_1), \dots, (z_n,y_n)$, and let $P$ denote their empirical distribution.
Throughout this section, we assume that the covariates $z_1, \dots, z_n$ are equipped with a total order, and that the indices are chosen such that $z_1 < z_2 < \dots < z_n$.
Repeated observations can also be easily accommodated as explained in Remark \ref{rem:3.1} below.

We aim to find an increasing function $g\colon \{z_1, \dots z_n\} \to \bar{\R}$ that minimizes
\begin{align}\label{eq:3}
\E_{P} S_\eta (g(Z),Y) \quad \text{for all } \eta \in \R,
\end{align}
where the random vector $(Z,Y)$ has distribution $P$.
Any increasing function $\hat{g}$ solving this optimization problem is a solution to the isotonic regression problem that is optimal with respect to all scoring functions in the class $\mathcal{S}$, simultaneously.

Condition \eqref{eq:3} is equivalent to minimizing $\E_{P}\one\{\eta \leq g(Z)\}V(\eta,Y)$ for all $\eta \in \R$.
We can rephrase the minimization problem to reflect the way in which we prove the main result:
For a given $\eta \in \R$, we have to find an index $i \in \{1, \dots, n + 1\}$ that minimizes 
\[
s_{i}(\eta) = v_{i:n}(\eta) = \sum_{\ell = i}^n V(\eta, y_\ell).
\]
Thereby, we obey the condition
\[
\{z : g(z) \ge \eta\} = \{z_i, \dots, z_n\},
\]
implied by the monotonicity constraint on $g$.
This index search needs to be conducted for every $\eta \in \R$ separately.
In a nutshell, we find the generalized inverse to an optimal solution.
Afterwards, we define the overall minimizing function $\hat{g}$.

From now on, we assume that all indices $i,j \in \{1, \dots, n+1\}$ unless specified otherwise.

\begin{remark}\label{rem:3.1}
	The assumption that the ordering $z_1 < \dots < z_n$ is strict is non-restrictive.
	Given a series of observations $(z_1', y_{1}), \dots, (z_m', y_m)$ with non-strictly ordered or unordered $z_i'$, we can choose $z_1 < \dots < z_n < z_{n + 1} = \infty$ such that $\{z_1', \dots, z_m'\} \subseteq \{z_1, \dots, z_n\}$.
	We define the empirical counting measure for the index range from $i$ to $j$ by
	\[
	P_{i:j}(B) = \sum_{\ell = i}^j \sum_{k = 1}^m \one\{z_\ell = z_k'\} \one\{(z_k', y_k) \in B\},
	\]
	with the corresponding integral of the identification function being equal to the following sum,
	\[
	V(\eta, P_{i:j}) = v_{i:j}(\eta) = \sum_{\ell = i}^j \sum_{k = 1}^m \one\{z_\ell = z_k'\} V(\eta, y_k).
	\]
	For condition \eqref{eq:3}, we write the empirical probability distribution as $P(B) = P_{1:n}(B) / m$.
	The subsequent arguments leading to an optimal solution rely solely on the identification sum $v_{i:j}(\eta)$ and the functional $T(P_{i:j})$, where we dealt with the dependence on the number of observations for each unique value of $z_\ell$ in the above generalization.
\end{remark}

We begin by introducing sets consisting of minimizing indices.
For $\eta \in \R$, let $I(\eta)$ denote the set of indices $i$ minimizing $s_{i}(\eta)$, and define $\mathcal{I}=\bigcup_{\eta \in \R} I(\eta) \subseteq \{1, \dots, n+1\}$.
That is, $i \in \mathcal{I}$ if and only if there exists an $\eta \in \R$ such that
\begin{equation*}
s_{i}(\eta) \le s_{j}(\eta) \quad \text{for all } j.
\end{equation*} 
The following proposition is immediate.
\begin{proposition}\label{prop:immediate}
	Let $\eta \in \R$.
	The inclusion $i \in I(\eta)$ holds if and only if, 
	\begin{align*}
	v_{i:(j-1)}(\eta) \le 0 \quad \text{for all $j > i$},\\
	v_{j:(i-1)}(\eta) \ge 0 \quad \text{for all $j < i$}.
	\end{align*}
	If $j > i$ and $v_{i:(j-1)}(\eta) = 0$, then  $j \in I(\eta)$. Analogously, if $j < i$ and $v_{j:(i-1)}(\eta) = 0$, then $j \in I(\eta)$.
\end{proposition}

The following proposition is a key observation to show optimality of the min-max and max-min solution.
We relate the threshold $\eta \in \R$ to the minimal and maximal elements of the functional $T$ on subsets of the data.
We write $T_{i:j}^- = T_{P_{i:j}}^- = \inf T(P_{i:j})$ and $T_{i:j}^+ = T_{P_{i:j}}^+ = \sup T(P_{i:j})$.
\begin{proposition}\label{prop:simplified_1}
	Let $\eta \in \R$, and $i \in I(\eta)$.
	Then,
	\begin{align*}
	\max_{j < i} T_{j:(i - 1)}^-
	&\le	\eta
	\le	\min_{j > i} T_{i:(j -1)}^+,\\
	\max_{j < i, j \not\in I(\eta)} T_{j:(i - 1)}^+
	&< \eta
	\le \min_{j > i, j \not\in I(\eta)} T_{i:(j -1)}^-.
	\end{align*}
\end{proposition}
\begin{proof}
	For all $j < i$, we have $v_{j:(i - 1)}(\eta) \ge 0$.
	For all $j > i$, we have $v_{i:(j - 1)}(\eta) \le 0$.
	Both inequalities are strict when $j \notin I(\eta)$.
	Corollary \ref{cor:TviaV} implies the result.
\end{proof}
Figure \ref{fig:etaT} illustrates the statement in Proposition \ref{prop:simplified_1}.

\begin{figure}
	\centering
	\includegraphics{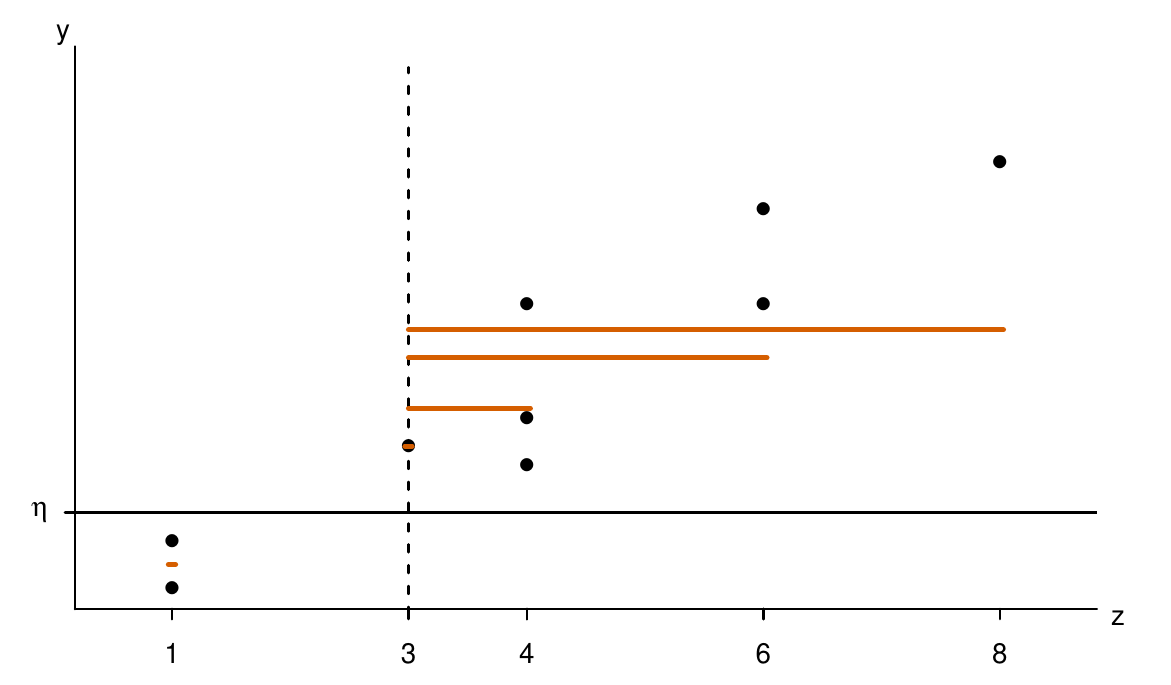}
	\caption{\textbf{Minimizing indices are separators.} For a sample of 9 data points, the graph illustrates the functional value (expectation) on relevant subsets of the data for a given $\eta$ and the minimizing index $i = 3$.
		The expectation value (vertical location of a brown line) is above or below $\eta$ when the corresponding subsample extends (horizontal extension of a brown line) to the right or left of the minimizing index, respectively \label{fig:etaT}}
\end{figure}

\begin{lemma} \label{lemma:iota}
	\begin{enumerate}[label=(\alph*)]
		\item Let $\eta, \eta' \in \R$, $\eta<\eta'$, and $i\in I(\eta)$,
		$j\in I(\eta')$. Then, $\min\{i,j\} \in I(\eta)$ and 
		$\max\{i,j\} \in I(\eta')$.
		\item For all $\eta \in \R$, $I(\eta)$ is a set of consecutive 
		indices in $\mathcal{I}$.
		In other words, if $i,j \in I(\eta)$ and $i<i_0 < j$ such that
		$i_0 \notin I(\eta)$, then $i_0 \notin \mathcal{I}$.
		\item The functions
		\[
		\eta \mapsto \min I(\eta) \quad \text{and} \quad \eta 
		\mapsto \max I(\eta)
		\]
		are increasing.
		
		\item Suppose that $\eta_m \uparrow \eta$ and 
		$i \in I(\eta_m)$ for all $m \in \mathbb{N}$.
		Then, $i \in I(\eta)$.
		\label{lemma:iota-c}
	\end{enumerate}
\end{lemma}
\begin{proof}
	\begin{enumerate}[label=(\alph*)]
		\item If $j \ge i$ the statement is trivial. Otherwise,
		$v_{j:(i-1)}(\eta) \geq 0 \geq v_{j:(i-1)}(\eta')$ by Proposition
		\ref{prop:immediate}, hence $v_{j:(i-1)}(\eta'') = 0$ for all 
		$\eta'' \in [\eta, \eta']$ due to the monotonicity of the
		identification function. The statement follows from Proposition
		\ref{prop:immediate}.
		\item Suppose the contrary: There exists an $\eta' \ne \eta$ such 
		that $i_0 \in I(\eta')$.
		If $\eta' < \eta$, we have that $v_{i:(i_0 - 1)}(\eta') \ge 0$.
		Similarly, since $i_0 \notin I(\eta)$ it holds that 
		$v_{i:(i_0-1)}(\eta) < 0$.
		This contradicts the monotonicity assumption for the first 
		argument of $V$.
		The argument against an $\eta' > \eta$ such that 
		$i_0 \in I(\eta')$ works similarly.
		
		\item Let $\eta < \eta'$ and suppose the contrary: Let 
		$i = \min I(\eta)$ and $i' = \min I(\eta')$ such that $i > i'$.
		Then, $i' \not\in I(\eta)$, and we have 
		$v_{i':(i - 1)}(\eta') \le 0$ and $v_{i':(i - 1)}(\eta) > 0$, 
		contradicting the monotonicity assumption for the first 
		argument of $V$.
		The argument for $\eta \mapsto \max I(\eta)$ works similarly.
		
		\item
		The left-continuity of the identification function implies that 
		$s_{j}(\eta_m) \uparrow s_{j}(\eta)$ as $\eta_m \uparrow \eta$ 
		for all $j$.
		If $i \in I(\eta_m)$ for all $m \in \mathbb{N}$, then 
		$s_{i}(\eta_m) \le s_{j}(\eta_m)$ for all $m \in \mathbb{N}$ and 
		all $j$, which in combination with the left-continuity implies
		$s_{i}(\eta) \le s_{j}(\eta)$ for all $j$.
	\end{enumerate}
\end{proof}

Lemma \ref{lemma:iota} confirms the existence of a left-continuous function $\iota\colon \R \to \{1, \dots, n+1\}$ mapping $\eta$ to a score-minimizing index $i$ that indicates the smallest $z_\ell$ in the corresponding set $\{z_i, \dots, z_n\}$.
Note that $\lim_{\eta \to -\infty}\iota(\eta) = 1$ and $\lim_{\eta \to \infty}\iota(\eta) = n+1$.

Then, any function $g\colon \{z_1, \dots, z_n\} \to \R$ with superlevel sets corresponding to the sets induced by $\iota$, i.e., with $g(z_\ell) \ge \eta$ for all $z_\ell \in \{z_{\iota(\eta)}, \dots, z_n\}$ for all $\eta \in \R$, must be an optimizing solution.
In fact, this solution is unique for a given $\iota$ because monotone functions are characterized by their superlevel sets.
\begin{proposition} \label{prop:simplified_2}
	Let $\iota\colon \R \to \{1, \dots, n+1\}$ be an increasing, left-continuous function such that $\iota(\eta) \in I(\eta)$.
	Then, the function $\hat{g}\colon \{z_1, \dots, z_n\} \to \R$ given by
	\begin{equation}\label{eq:hatg}
	\inf\{\eta : \iota(\eta) > \ell\}
	= \hat{g}(z_\ell)
	= \max\{\eta: \iota(\eta) \le \ell \}
	\end{equation}
	is the unique function that satisfies
	\[
	\{z : g(z) \geq \eta\} = \{z_{\iota(\eta)}, \dots, z_n\} \quad \text{ for all } \eta \in \R,
	\]
	among all increasing functions $g\colon \{z_1, \dots, z_n\} \to \R$.
\end{proposition}
\begin{proof}
	Due to the monotonicity and left-continuity of $\iota \colon \R \to \{1, \dots, n+1\}$, we have
	$\inf\{\eta : \iota(\eta) > \ell\}
	= \max\{\eta: \iota(\eta) \le \ell \}$, $\ell =1, \dots, n$.
	The monotonicity of $\hat{g}$ follows from the monotonicity of $\iota$ and the fact that $\{z_1, \dots, z_n\}$ is ordered.
	Let $\eta' \in \R$.
	Then,
	\begin{align*}
	& \text{(i)} \quad \hat{g}(z_\ell) \ge \eta' \implies
	\iota(\hat{g}(z_\ell)) \ge \iota(\eta') \implies
	\ell \ge \iota(\eta'),
	\\
	& \text{(ii)} \quad \hat{g}(z_{\iota(\eta')}) 
	= \max \{\eta : \iota(\eta) = \iota(\eta')\} \ge \eta'.
	\end{align*}
	Therefore, $\{z : \hat{g}(z) \ge \eta'\} \subseteq \{z_{\iota(\eta')}, \dots, z_n\} \subseteq \{z : \hat{g}(z) \ge \eta'\}$ where the first inclusion follows by (i) and the second by (ii).
	Uniqueness follows because any hypothetical alternative $\bar{g}$ with $\bar{g}(z_\ell) \neq \hat{g}(z_\ell)$ for some $\ell \in \{1, \dots, n\}$ leads to the contradiction $\{z_{\iota(\eta)}, \dots, z_n\} = \{z : \bar{g}(z) \geq \eta\} \neq \{z : \hat{g}(z) \geq \eta\} = \{z_{\iota(\eta)}, \dots, z_n\}$ for all $\eta$ between $\bar{g}(z_\ell)$ and $\hat{g}(z_\ell)$.
\end{proof}

In Figure \ref{fig:ghat} we give an example for 6 data points.
The example illustrates how the values $\hat{g}(z_\ell)$, $\ell = 1, \dots, n$, can be determined from the epigraph of the function $\eta \mapsto z_{\iota(\eta)}$.

\begin{figure}
	\centering \includegraphics{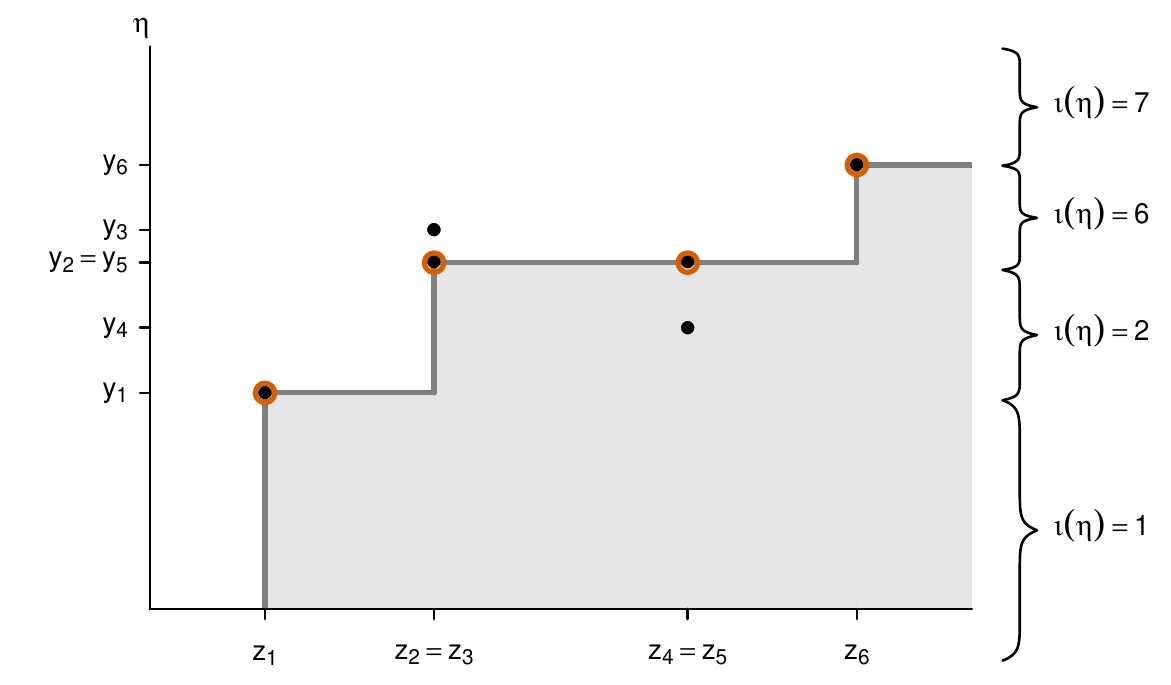}
	\caption{\textbf{Graph of $\hat{g}$.} For a sample of 6 data points, the values of $\hat{g}(z)$ for $z = z_1, \dots, z_6$ are shown in red.
		The epigraph of the function $\eta \mapsto z_{\iota(\eta)}$ is shown in grey, where $T$ is chosen as the median functional to choose $\iota(\eta)$ \label{fig:ghat}}
\end{figure}

Now, we can state and show one of our main results which is that $\hat{g}$ coincides with or is bounded by a min-max and max-min solution.

\begin{proposition}\label{prop:simplified_main1}
	Let $\ell \in \{1, \dots, n\}$ and let $\hat{g}$ be a solution to the isotonic regression problem.
	Then, 
	\[
	\min_{j \geq \ell} \, \max_{i \leq j} T_{i:j}^-
	\leq \hat{g}(z_\ell)
	\leq \max_{i \leq \ell} \, \min_{j \geq i} T_{i:j}^+.
	\]
\end{proposition}
\begin{proof}
	Applying the first set of bounds from Proposition \ref{prop:simplified_1} to the formula for $\hat{g}$ at \eqref{eq:hatg} yields
	\[
	\inf_{\iota(\eta) > \ell} \, \max_{i < \iota(\eta)} T_{i:(\iota(\eta) - 1)}^-
	\le \hat{g}(z_\ell)
	\le \max_{\iota(\eta) \le \ell} \, \min_{j > \iota(\eta)} T_{\iota(\eta):(j - 1)}^+.
	\]
	The lower bound is bounded below by $\min_{j \ge \ell} \, \max_{i \le j} T_{i:j}^-$, and the upper bound is bounded above by $\max_{i \le \ell} \min_{j \ge i} T_{i:j}^+$.
\end{proof}

The max-min inequality implies that for functionals $T$ of singleton type, e.g., the expectation or expectile functionals, the lower and upper bound in Proposition \ref{prop:simplified_main1} are equal.
In general, a similar statement always holds, where the choice of $\iota$ determines whether $\hat{g}$ attains the minimal or maximal elements of the functional.
\begin{proposition} \label{prop:simplified_main2}
	Let $\ell \in \{1, \dots, n\}$.
	\begin{enumerate}[label=(\alph*)]
		\item If $\iota(\eta)=\min I(\eta)$ for all $\eta \in \R$, then,
		\[
		\hat{g}(z_\ell)  
		= \min_{j \geq \ell} \, \max_{i \leq j} T_{i:j}^+ 
		= \max_{i \leq \ell} \, \min_{j \geq i} T_{i:j}^+.
		\]
		
		\item If $\iota(\eta)=\max I(\eta)$ for all $\eta \in \R$, then,
		\[
		\hat{g}(z_\ell)  
		= \min_{j \geq \ell} \, \max_{i \leq j} T_{i:j}^-
		= \max_{i \leq \ell} \, \min_{j \geq i} T_{i:j}^-.
		\]
	\end{enumerate}
\end{proposition}
\begin{proof}
	The proof works the same way as the proof of Proposition \ref{prop:simplified_main1} but using second set of bounds in Proposition \ref{prop:simplified_1}.
	This is possible because in (a) we have that for $j < \iota(\eta)$ it holds that $j \not\in I(\eta)$ and in (b), for $j > \iota(\eta)$ we know that $j \not\in I(\eta)$.
\end{proof}

Let us denote the solution in part (a) of Proposition \ref{prop:simplified_main2} by $g^+$ and the one in part (b) by $g^-$.
Clearly, it always holds that $g^- \le g^+$.
It is a natural question whether any increasing function $g$ that satisfies $g^- \le g \le g^+$ is also a minimizer of the criterion \eqref{eq:3}.
It turns out that the answer is negative; see \citet[Remark 2.2, Example 2.4]{MoschingDumbgen2019}.
The following proposition provides a simple sufficient criterion for $g$ to also be a solution.
In the case of quantiles and for the classical asymmetric linear loss, the same result is shown in \citet[Lemma 2.1]{MoschingDumbgen2019}.
Note that in Proposition \ref{prop:suffcrit} it is not required that $g^-$, $g^+$ are the solutions from Proposition \ref{prop:simplified_main2} as long as they satisfy $g^- \le g^+$.

\begin{proposition}\label{prop:suffcrit}
	Let $g^-, g^+$ be two solutions to the isotonic regression problem, that is, minimizers of \eqref{eq:3}, and suppose they satisfy $g^- \le g^+$.
	Let $\hat{g}$ be increasing, $g^- \le \hat{g}\le g^+$, and suppose that $g^+(z_\ell) = g^+(z_{\ell'})$, $g^-(z_\ell) = g^-(z_{\ell'})$ for some $\ell < \ell'$ implies $\hat{g}(z_\ell) = \hat{g}(z_{\ell'})$.
	Then, $\hat{g}$ is also a minimizer of \eqref{eq:3}, that is, a solution to the isotonic regression problem.
\end{proposition}
\begin{proof}
	For $\eta \in \R$, define $\iota(\eta):= \min\{\ell : \hat{g}(z_\ell) \ge \eta\}$, and analogously $\iota^-(\eta)$ and $\iota^+(\eta)$ with $\hat{g}$ replaced by $g^-$ and $g^+$, respectively.
	The functions $\iota$, $\iota^-$, $\iota^+$ are increasing and left-continuous.
	For $\iota^-$, $\iota^+$ it holds that $\iota^-(\eta), \iota^+(\eta) \in I(\eta)$.
	For all $\ell \in \{1, \dots, n\}$, we have that
	\begin{align*}
	g^-(z_\ell) &= \max\{\eta : \iota^-(\eta) \le \ell\}\\ &\le \hat{g}(z_\ell) = \max\{\eta : \iota(\eta) \le \ell\} \le g^+(z_\ell) = \max\{\eta : \iota^+(\eta) \le \ell\},
	\end{align*}
	therefore, $\iota^+(\eta) \le \iota(\eta) \le \iota^-(\eta)$ for all $\eta \in \R$.
	By Lemma \ref{lemma:iota} (b), it remains to show that $\iota(\eta) \in \mathcal{I}$.
	This follows from the following two observations.
	
	First, if 
	\[
	\hat{g}(z_\ell) = \max\{\eta : \iota(\eta) \le \ell\} = \max\{\eta : \iota(\eta) \le \ell'\}= \hat{g}(z_{\ell'}) 
	\]
	for some $\ell \le \ell'$, then $\iota(\eta) \not\in\{\ell+1,\dots,\ell'\}$.
	Second, if $g^-(z_\ell) < g^-(z_{\ell+1})$ or $g^+(z_\ell) < g^+(z_{\ell+1})$, then $\ell + 1 \in \mathcal{I}$.
\end{proof}

The proof of Proposition \ref{prop:suffcrit} shows that $\hat g$ may jump at points $z_\ell$ where $g^+$ and $g^-$ do not jump as long as $\ell \in \mathcal{I}$, that is, as long as $\ell$ is a minimizing index for some $\eta$.
The following Proposition \ref{prop:immediate2} characterizes the possible additional jumps of $\hat g$.

\begin{proposition} \label{prop:immediate2}
	Let $g^-, g^+$ be two solutions to the isotonic regression problem, and suppose that for some $i,j \in \{1, \dots, n\}$, $i<j$,
	\[
	\eta^-:= g^-(z_i) = g^-(z_j) <  g^+(z_i) = g^+(z_j) =:\eta^+.
	\]
	Furthermore, assume it holds that for $i > 1$, $g^-(z_{i-1}) \not=g^-(z_i)$ or $g^+(z_{i-1}) \not=g^+(z_i)$, and for $j < n$, $g^-(z_{j}) \not=g^-(z_{j+1})$ or $g^+(z_{j}) \not=g^+(z_{j+1})$.
	Then for $\ell \in \{i+1, \dots, j\}$, we have $T_{i:(\ell-1)}^- \le \eta^-$ if and only if $\ell\in I(\eta)$ for all $\eta \in (\eta^-, \eta^+]$.
\end{proposition}
\begin{proof}
	We will first argue that $T_{i:k}^+ \ge \eta^+$ for $k > i$, and that $T_{k:j}^- \le \eta^-$ for $k < j$.
	The assumptions ensure that there are $i^- \le i$, $i^+\le i$ such that $i^- \in I(\eta^-)$, $i^+ \in I(\eta^+)$, and $\max\{i^-,i^+\} = i$. By Lemma \ref{lemma:iota} (a) we have $i \in I(\eta^+)$, and therefore $v_{i:k}(\eta^+)\le 0$ for all $k > i$ by Proposition \ref{prop:immediate}, hence $T_{i:k}^+ \ge \eta^+$ by Corollary \ref{cor:TviaV}.
	
	The assumptions also imply that there are $\tilde{\eta}^- > \eta^-$, $j^- \ge j$, $\tilde{\eta}^+ > \eta^+$, $j^+ \ge j$ such that $j^- + 1\in I(\eta)$ for all $\eta \in (\eta^-,\tilde{\eta}^-]$, $j^+ + 1\in I(\eta)$ for all $\eta \in (\eta^+,\tilde{\eta}^+]$, and $\min\{j^-, j^+\} = j$.
	By Lemma \ref{lemma:iota} (a) we have $j+1 \in I(\eta)$ for $\eta \in (\eta^-,\tilde{\eta}^-]$, and therefore $v_{k:j}(\eta)\ge 0$ for all $k \le j$, $\eta \in (\eta^-,\tilde{\eta}^-]$, hence $T_{k:j}^- \le \eta$.
	In summary, $T_{k:j}^- \le \eta^-$ for all $k \le j$.
	
	For the first part of the result, let $\ell \in \{i+1, \dots, j\}$ such that $T_{i:(\ell-1)}^- \le \eta^-$.
	By Lemma \ref{lemma:functional_bounds}, we have $T_{i:k}^+ \le \max\{ T_{i:(\ell-1)}^-, T_{\ell:k}^+ \}$ for all $k \ge \ell$.
	Since $T_{i:k}^+ \ge \eta^+$ and $T_{i:(\ell-1)}^- \le \eta^-$, we have $T_{\ell :k}^+ \ge \eta^+$ and $v_{\ell:k}(\eta) \le 0$ for all $\eta \le \eta^+$, $k \ge \ell$.
	Similarly, by Lemma \ref{lemma:functional_bounds}, we have $T_{k:j}^- \ge \min\{ T_{k:(\ell-1)}^-, T_{\ell:j}^+ \}$ for all $k \le \ell-1$.
	Since $T_{k:j}^- \le \eta^-$ and as shown above $T_{\ell:j}^+ \ge \eta^+$, we have $T_{k:(\ell-1)}^- \le \eta^-$ and $v_{k:(\ell-1)}(\eta) \ge 0$ for all $\eta > \eta^-$, $k \le \ell-1$.
	Hence, we have $\ell  \in I(\eta)$ for all $\eta \in (\eta^-, \eta^+]$.
	
	To prove the converse, note that $\ell \in I(\eta)$ for all $\eta \in (\eta^-, \eta^+]$ implies $v_{k:(\ell-1)}(\eta) \ge 0$ for all $\eta \in (\eta^-, \eta^+]$, $k < \ell$.
	Hence, in particular, $v_{i:(\ell-1)}(\eta) \ge 0$ and $T_{i:(\ell-1)}^- \le \eta$ for all $\eta \in (\eta^-, \eta^+]$, and, therefore, $T_{i:(\ell-1)}^- \le \eta^-$.
\end{proof}

\subsection{Pool-adjacent-violators algorithm}\label{sec:PAV}
As in Section \ref{sec:minmax}, the PAV algorithm takes observations $(z_1, y_1), \dots, (z_n, y_n)$, with $z_1 < \dots < z_n$ and can be generalized as detailed in Remark \ref{rem:3.1}.
Its starting point is the finest partition $\mathcal{Q}_0 = \{\{1\}, \dots, \{n\}\}$ of the index set $\{1, \dots, n\}$, and a corresponding function $g_0 \colon \{z_1, \dots, z_n\} \to \R$ satisfying
\[
g_0(z_\ell) \in T(P_{\ell:\ell}).
\]
If possible, an increasing function has to be chosen.
The algorithm iteratively considers pooling adjacent elements $Q_1$ and $Q_2$ in the current partition, where ``adjacent'' means that the largest element of $Q_1$, $Q_1^+ = \max Q_1$, is the predecessor (in terms of the natural numbers) of the smallest element of $Q_2$, $Q_2^- = \min Q_2$.
Pooling adjacent partition elements is considered necessary when $T_{Q_1^-:Q_1^+}^- > T_{Q_2^-:Q_2^+}^+$ (strong adjacent violators), it is considered invalid when $T_{Q_1^-:Q_1^+}^+ < T_{Q_2^-:Q_2^+}^-$, and optional otherwise (weak adjacent violators).
The early stopping criterion is the existence of an increasing function $g_\mathrm{PAV} \colon \{z_1, \dots, z_n\} \to \R$ that is constant on each element of the current partition $\mathcal{Q}_\mathrm{PAV}$ and satisfies
\begin{equation} \label{eq:PAV1}
g_\mathrm{PAV}(z_\ell) \in T(P_{Q^-:Q^+}) \quad \text{for all $Q \in \mathcal{Q}_\mathrm{PAV}$ and $\ell \in Q$},
\end{equation}
that is, when no further pooling is necessary.
The late stopping criterion is reached when no weak adjacent violators remain.
The first and most apparent property we observe is that for all $\ell \in \{1, \dots, n\}$, $Q_1, Q_2 \in \mathcal{Q}_\mathrm{PAV}$, $Q_1^- \le \ell \le Q_2^+$, we have
\begin{equation} \label{eq:PAV2}
T_{Q_1^-:Q_1^+}^- \le g_\mathrm{PAV}(z_\ell) \le T_{Q_2^-:Q_2^+}^+,
\end{equation}
since otherwise either $g_\mathrm{PAV}$ is not increasing or the condition \eqref{eq:PAV1} is violated.
Definition \ref{def:TviaV} and its Corollary \ref{cor:TviaV} allow for an immediate proof of an additional property of $\mathcal{Q}_\mathrm{PAV}$.
\begin{proposition} \label{prop:nec-pooling}
	Let $\mathcal{Q}$ be a partition of $\{1, \dots, n\}$ found by the PAV algorithm, $Q \in \mathcal{Q}$, and $j \in Q$.
	Then,
	\begin{equation*}
	T_{j:Q^+}^- \le T_{Q^-:Q^+}^- \le T_{Q^-:Q^+}^+ \le T_{Q^-:j}^+.
	\end{equation*} 
\end{proposition}
\begin{proof}
	The second inequality is trivial.
	For the first inequality, suppose the contrary: There exist $\eta \in \R$, $j \in Q$ such that $T_{Q^-:Q^+}^- < \eta < T_{j:Q^+}^-$.
	This implies that $j > Q^-$ and $v_{Q^-:Q^+}(\eta) \ge 0 > v_{j:Q^+}(\eta)$, hence $v_{Q^-:(j - 1)}(\eta) > 0$.
	Therefore, $T_{Q^-:(j - 1)}^+ < \eta < T_{j:Q^+}^-$, which means that $Q$ can be seen as the result of an invalid pooling of $\{Q^-, \dots, j - 1\}$ and $\{j, \dots, Q^+\}$.
	A similar argument applies to the third inequality.
\end{proof}
To show the connection between a valid solution by the PAV algorithm and the score optimizing solution $\hat{g}$ in Section \ref{sec:minmax}, we define
\begin{equation} \label{eq:iotaPAV}
\iota_\mathrm{PAV}(\eta) = \min \{k : \eta \le g_\mathrm{PAV}(z_k)\}.
\end{equation}
Plugging $\iota_\mathrm{PAV}$ into the definition of $\hat{g}$ recovers $g_\mathrm{PAV}$,
\begin{align*}
\hat{g}(z_\ell) &= \max \{ \eta : \iota_\mathrm{PAV}(\eta) \le \ell\} \\
&= \max\{\eta : \eta \le g_\mathrm{PAV}(z_\ell) \} = g_\mathrm{PAV}(z_\ell).
\end{align*}
In order to show that $g_\mathrm{PAV}$ solves the isotonic regression problem, it remains to be shown that $\iota_\mathrm{PAV}(\eta) \in I(\eta)$ for all $\eta \in \R$.
\begin{proposition}
	Let $\eta \in \R$, then $\iota_\mathrm{PAV}(\eta) \in I(\eta)$.
\end{proposition}
\begin{proof}
	Let $\eta \in \R$.
	We combine Proposition \ref{prop:nec-pooling}, the statement \eqref{eq:PAV2}, and the defining equation \eqref{eq:iotaPAV}.
	As a result, for all $j, k \in \{1, \dots, n + 1\}$, $j < \iota_\mathrm{PAV}(\eta) < k$, we have $T_{j:(\iota_\mathrm{PAV}(\eta) - 1)}^- \le g_\mathrm{PAV}(z_{\iota_\mathrm{PAV}(\eta) - 1}) < \eta \le g_\mathrm{PAV}(z_{\iota_\mathrm{PAV}(\eta)}) \le T_{\iota_\mathrm{PAV}(\eta):(k - 1)}^+$, hence $v_{j:(\iota_\mathrm{PAV}(\eta) - 1)}(\eta) \ge 0 \ge v_{\iota_\mathrm{PAV}(\eta):(k - 1)}(\eta)$.
	The statement follows from Proposition \ref{prop:immediate}.
\end{proof}

As a closing side note, we point out that $\iota_\mathrm{PAV}$ corresponds to coarsest partition that allows the solution $g_\mathrm{PAV}$.
Any weak adjacent violators on which $g_\mathrm{PAV}$ takes the same value have been pooled.

\section{Generalization to partial orders} \label{sec:partial}

In the first part of this paper, we considered a series of observations $(z_\ell, y_{\ell})$, where $\ell =1, \dots, n$ and the set $\{z_1, \dots, z_n\}$ was totally ordered.
In this section, we solve the isotonic regression problem considering a distribution $P$ for a random vector $(Z, Y) \in \mathcal{Z} \times \R$, where $\mathcal{Z}$ is a finite partially ordered set.
Analogously to \eqref{eq:3}, we aim to minimize the criterion
\begin{equation}\label{eq:3A}
\E_{P} S_\eta (g(Z),Y) \quad \text{for all } \eta \in \R,
\end{equation}
over all increasing functions $g\colon \mathcal{Z} \to \bar\R$.
We call any minimizer of \eqref{eq:3A} a solution to the isotonic regression problem.

The considerations in Section \ref{sec:minmax} lead to the reformulation of the optimization problem at \eqref{eq:3} as the minimization of
\[
s_{i}(\eta) = v_{i:n}(\eta) = \sum_{\ell = i}^n V(\eta, y_\ell)
\]
over all $i \in \{1, \dots, n+1\}$.
The dependence on $z_1, \dots, z_n$ and $\hat{g}$ seemingly vanishes, but remains encoded in the index set $\{1, \dots, n+1\}$ and in the link to $\eta$ via an optimizing function $\iota \colon \R \to \{1, \dots, n+1\}$ such that
\[
\{z : \hat{g}(z) \ge \eta\} = \{z_{\iota(\eta)}, \dots, z_n \}.
\]
In the second part, we now generalize the index set $\{1, \dots, n+1\}$ and the function $\iota$ in order to accommodate a partially ordered set $\mathcal{Z}$.

We introduce upper sets $x \subseteq \mathcal{Z}$ to replace single indices $i \in \{1, \dots, n + 1\}$.
We consider a set $\mathcal{X} \subseteq \mathcal{P}(\mathcal{Z})$, where $\mathcal{P}$ denotes the power set.
The set $\mathcal{X}$ consists of all admissible superlevel sets for an increasing function $g$ imposed by the partial order on $\mathcal{Z}$.
A set $x \in \mathcal{X}$ is characterized by the property that if $z \in x$ and $z \preceq z'$, then $z' \in x$.
This implies that $\mathcal{X}$ is a finite lattice, that is, it is closed under union and intersection and contains $\mathcal{Z}$ and the empty set.

Consequently, we replace the function $\iota \colon \R \to \{1, \dots, n+1\}$ with a function $\xi \colon \R \to \mathcal{X}$, that maps $\eta$ to an upper set $x$ of $\mathcal{Z}$ that minimizes
\begin{equation}\label{eq:generalmin}
s_x(\eta) = v_x(\eta) = V(\eta, P_x) = \int_{x \times \R} V(\eta, y)\, P(\mathrm{d}z, \mathrm{d}y),
\end{equation}
where $P_x(A) = P((x \times \R) \cap A)$ for any $A \in \mathcal{P}(\mathcal{Z}) \otimes \mathcal{B}(\R)$, where $\mathcal{B}(\R)$ denotes the Borel $\sigma$-algebra on $\R$.
In this notation, $s_x$ is only defined for $x \in \mathcal{X}$, whereas $v_x$ and $P_x$ are defined for any $x \in \mathcal{P}(\mathcal{Z})$.
Let $X(\eta)$ denote the set of superlevel sets $x \in \mathcal{X}$ minimizing $s_x(\eta)$ at \eqref{eq:generalmin}.
Since $\mathcal{P}(\mathcal{Z})$ is finite, such a minimizer always exists.

The following proposition is the analogue to Proposition \ref{prop:immediate} and gives necessary and sufficient conditions for the inclusion $x \in X(\eta)$.
\begin{proposition}\label{prop:4.1}
	Let $\eta \in \R$.
	Subject to $x,x' \in \mathcal{X}$, the inclusion $x \in X(\eta)$ holds if and only if
	\begin{align*}
	&v_{x\setminus x'} (\eta) \leq 0 
	\quad \text{for all } x' \subsetneq x, \\&
	v_{x' \setminus x}(\eta) \geq 0 
	\quad \text{for all } x' \supsetneq x.
	\end{align*}
	Let $x \in X(\eta)$, $x' \in \mathcal{X}$.
	If $v_{x\setminus x'}(\eta)= v_{x' \setminus x}(\eta)$, then $x' \in X(\eta)$.
\end{proposition}
\begin{proof}
	Note that $s_x(\eta) \leq s_{x'}(\eta)$ for all $x' \subsetneq x$ and all $x' \supsetneq x$ holds if and only if $v_{x\setminus x'}(\eta) \leq 0$ for all $x' \subsetneq x$ and $v_{x' \setminus x} (\eta) \geq 0$ for all $x' \supsetneq x$.
	For the first part of the result, note that $x\in X(\eta)$ implies $s_x(\eta) \leq s_{x'}(\eta)$ for all $x' \subsetneq x$ and all $x' \supsetneq x$.
	Conversely, let $x \in \mathcal{X}$ be such that the latter condition is satisfied.
	Then, $s_x(\eta) \leq s_{x' \cap x}(\eta)$ and $s_x(\eta) \leq s_{x' \cup x}(\eta)$ for all $x' \in \mathcal{X}$.
	By substracting $v_{x\setminus x'}(\eta)$ on both sides of the latter inequality, we have $s_{x \cap x'}(\eta) \leq s_{x'}(\eta)$ for all $x'\in \mathcal{X}$, and hence $x \in X(\eta)$.
	The second part of the result is immediate after adding $s_{x \cap x'}(\eta)$ to both sides of $v_{x\setminus x'}(\eta)= v_{x' \setminus x}(\eta)$, that is, $s_x(\eta)=s_{x'}(\eta)$.
\end{proof}

\begin{corollary}\label{cor:4.1}
	Let $\eta \in \R$ and $x \in X(\eta)$, $x' \in \mathcal{X}$.
	If $x' \subsetneq x$ and $v_{x \setminus x'}(\eta)=0$, then $x' \in X(\eta)$.
	Analogously, if $x' \supsetneq x$ and $v_{x' \setminus x}(\eta) = 0$, then $x' \in X(\eta)$.
\end{corollary}

\begin{lemma} \label{lem:4.1}
	\begin{enumerate}[label = (\alph*)]
		\item
		Let $\eta, \eta' \in \R$, $\eta < \eta'$, and $x \in X(\eta)$, $x' \in X(\eta')$. Then, $v_{x' \setminus x}(\eta'') = 0$ for all $\eta'' \in [\eta, \eta']$.
		
		\item
		Let $\eta \in \R$ and $x', x'' \in X(\eta)$, $x \in \mathcal{X}$. If $x \in \bigcup_{\eta \in \R} X(\eta)$ and $x' \supseteq x \supseteq x''$, then $x \in X(\eta)$.
		
		\item
		Let $\eta, \eta' \in \R$, $\eta < \eta'$, and $x \in X(\eta)$, $x' \in X(\eta')$.
		Then, $x \cup x' \in X(\eta)$ and $x \cap x' \in X(\eta')$.
	\end{enumerate}
\end{lemma}
\begin{proof}
	\begin{enumerate}[label = (\alph*)]
		\item
		We have $(x \cup x') \setminus x = x' \setminus x = x' \setminus (x \cap x')$.
		The statement is trivial if $x' \setminus x = \emptyset$.
		Otherwise, $v_{x' \setminus x}(\eta) \ge 0 \ge v_{x' \setminus x}(\eta')$ by Proposition \ref{prop:4.1}, where the statement follows from the monotonicity of the identification function in its first argument.
		\item 
		The statement is trivial if $x = x'$, $x = x''$, or $x \notin X(\eta')$ for all $\eta' \neq \eta$.
		Therefore, assume $x \in X(\eta')$, $\eta' \neq \eta$.
		If $\eta < \eta'$, then $v_{x \setminus x''}(\eta) = 0$ by part (a).
		If $\eta' < \eta$, then $v_{x' \setminus x}(\eta) = 0$ by part (a).
		In either case, $x \in X(\eta)$ by Corollary \ref{cor:4.1}.
		\item
		We have $s_x(\eta) \le s_{x \cup x'}(\eta)$ and $s_{x'}(\eta') \le s_{x \cap x'}(\eta')$, and $v_{x' \setminus x}(\eta'') = 0$ for all $\eta'' \in [\eta, \eta']$ by part (a).
		That means, $s_x(\eta) = s_{x \cup x'}(\eta)$ and $s_{x'}(\eta') = s_{x \cap x'}(\eta')$.
	\end{enumerate}
\end{proof}

\begin{example}
	In the case $\mathcal{Z} = \{z_1, \dots, z_n\}$, we choose $\mathcal{X}$ as the image of  $\{1, \dots, n+1\}$ under the one-to-one mapping
	\begin{align*}
	i &\mapsto \{ z : z \ge z_i \},
	\end{align*}
	implying a total order on $\mathcal{Z}$.
	In combination with the function $\xi \colon \R \to \mathcal{X}$ defined by
	\begin{equation}\label{eq:corr}
	\xi(\eta) = \{ z: z \ge z_{\iota(\eta)} \},
	\end{equation}
	we can embed the results from Section \ref{sec:total} into the more general setting of a partial order on the covariates.
\end{example}

Equation \eqref{eq:corr} demonstrates that instead of an increasing function $\iota \colon \R \to \{1, \dots, n+1\}$ such that $\iota(\eta) \in I(\eta)$ for all $\eta \in \R$, we are now interested in a decreasing function $\xi \colon \R \to \mathcal{X}$ in the sense that for $\eta'>\eta$ it holds that $\xi(\eta')\subseteq \xi(\eta)$.
Furthermore, $\xi(\eta) \in X(\eta)$ should hold for all $\eta \in \R$.
In Proposition \ref{prop:generalized_2}, we show that the functions $\xi$ are in one-to-one correspondence to the solutions $\hat{g}$ of the isotonic regression problem at \eqref{eq:3A}, and in Lemma \ref{lemma:xi} we show existence of such a function $\xi$.

The following proposition is analogous to Proposition \ref{prop:simplified_2} and allows to recover $\hat{g}$ from $\xi$.
\begin{proposition} \label{prop:generalized_2}
	Let $\xi \colon \R \to \mathcal{X}$ be a decreasing, left-continuous function such that $\xi(\eta) \in X(\eta)$, where left-continuity means that if $\eta_n \uparrow \eta$ and $z \in \xi(\eta_n)$, then $z \in \xi(\eta)$.
	Then, the function $\hat{g} \colon \mathcal{Z} \to \R$ given by
	\begin{equation}\label{eq:hatg2}
	\inf\{\eta : z \notin \xi(\eta)\}
	= \hat{g}(z)
	= \max\{\eta: z \in \xi(\eta)\}
	\end{equation}
	is the unique function that satisfies
	\[
	\{z : g(z) \geq \eta\} = \xi(\eta) \quad \text{ for all } \eta \in \R,
	\]
	among all increasing functions $g \colon \mathcal{Z} \to \R$.
\end{proposition}
\begin{proof}
	The left-continuity and monotonicity of $\xi \colon \R \to \mathcal{X}$ implies equation \eqref{eq:hatg2}.
	The monotonicity of $\hat{g}$ follows from the monotonicity of $\xi$ and the fact that $\xi$ takes values being superlevel sets of the partial order on $\mathcal{Z}$.
	Let $\eta' \in \R$.
	Then,
	\begin{align*}
	\text{(i)}  &\quad \hat{g}(z) \geq \eta' \implies
	\xi(\hat{g}(z)) \subseteq \xi(\eta') \implies
	z \in \xi(\eta').
	\\
	\text{(ii)} &\quad \text{For any } z \in \xi(\eta'): \quad \hat{g}(z) = \max \{\eta : z \in \xi(\eta)\} \ge \eta'.
	\end{align*} 
	Therefore, $\{z : \hat{g}(z) \ge \eta'\} \subseteq \{z : z \in \xi(\eta')\} \subseteq \{z : \hat{g}(z) \ge \eta'\}$ where the first inclusion follows by (i) and the second by (ii).
	Uniqueness follows because any hypothetical alternative $\bar{g}$ with $\bar{g}(z') \neq \hat{g}(z')$ for some $z' \in \mathcal{Z}$ leads to the contradiction $\xi(\eta) = \{z : \bar{g}(z) \geq \eta\} \neq \{z : \hat{g}(z) \geq \eta\} = \xi(\eta)$ for all $\eta$ between $\bar{g}(z')$ and $\hat{g}(z')$.
\end{proof}

The following lemma guarantees the existence of a function $\xi$ as specified in Proposition \ref{prop:generalized_2}.
\begin{lemma} \label{lemma:xi}
	\begin{enumerate}[label = (\alph*)]
		\item
		There exists a decreasing function $\xi \colon \mathbb{Q} \to \mathcal{X}$ such that $\xi(q) \in X(q)$ for all $q \in \mathbb{Q}$.
		\item
		Let $\eta_n \uparrow \eta$ and $x_n \in X(\eta_n)$, $x_n \supseteq x_{n + 1}$.
		Then, $x = \bigcap_{n \in \mathbb{N}} x_n \in X(\eta)$.
	\end{enumerate}
\end{lemma}
\begin{proof}	
	\begin{enumerate}[label=(\alph*)]
		\item Let $\{q_n\} = \mathbb{Q}$ be an enumeration of the rationals.
		We define $\xi(q_n)$ inductively.
		Pick $x_1 \in X(q_1)$ and set $\xi(q_1) = x_1$.
		For $n \ge 2$, define 
		\[
		x_n^- = \bigcup_{\substack{i\in\{1,\dots,n-1\}\\q_i > q_n}} \xi(q_i), \quad x_n^+ = \bigcap_{\substack{i\in\{1,\dots,n-1\}\\q_i < q_n}} \xi(q_i),
		\] 
		if $\{i : q_i > q_n\} \not= \emptyset$ and $\{i : q_i < q_n\} \not= \emptyset$.
		If $\{i : q_i > q_n\} = \emptyset$, we set $x_n^- = \emptyset$, and if $\{i : q_i < q_n\} = \emptyset$, we set $x_n^+ = \mathcal{Z}$.
		We choose any $x_n \in X(q_n)$ and set $\xi(q_n) = (x_n \cup x_n^-) \cap x_n^+$.
		At each step $n$, $\xi(q_n) \in X(q_n)$ follows by \ref{lem:4.1} (a), and $\xi(q_n) \subseteq x_n^+$.
		Furthermore, we show by induction that $x_n^- \subseteq \xi(q_n)$ for all $n$.
		For $n=2$, this is easily verified.
		Suppose the claim holds for $n-1 \ge 2$.
		If $q_{n}> q_{n-1}$, then $x_n^- = x_{n-1}^-$ and $x_n^+ = x_{n-1}^+ \cap \xi(q_{n-1})=\xi(q_{n-1})$, hence
		\[
		x_n^- = x_{n-1}^- \subseteq (x_n \cup x_{n-1}^-) \cap \xi(q_{n-1}) = \xi(q_n).
		\]  
		If $q_n < q_{n-1}$, then $x_n^- = x_{n-1}^- \cup \xi(q_{n-1})=\xi(q_{n-1})$ and $x_n^+ = x_{n-1}^+$, hence
		\[
		x_n^- = \xi(q_{n-1}) \subseteq (x_n \cup \xi(q_{n-1})) \cap x_{n-1}^+ = \xi(q_n).
		\] 
		In summary, for $k < n$, if $q_k < q_n$, then $\xi(q_n) \subseteq x_n^+ \subseteq \xi(q_k)$, and  if $q_k > q_n$, $\xi(q_k) \subseteq x_n^- \subseteq \xi(q_n)$ showing that $\xi$ is decreasing.
		\item
		We have $s_{x_n}(\eta_n) \le s_{x'}(\eta_n)$ for all $x' \in \mathcal{X}$.
		Furthermore, the definitions of $x$ and $V$ imply $\one\{z \in x_n\}V(\eta_n, y) \rightarrow \one\{z \in x\}V(\eta, y)$ pointwise, and we have $\one\{z \in x_n\}V(\eta_n, y) \le \sup_{n \in \mathbb{N}} |V(\eta_n, y)|$.
		By the dominated convergence theorem, $s_{x_n}(\eta_n) \rightarrow s_x(\eta)$ and $s_{x'}(\eta_n) \rightarrow s_{x'}(\eta)$.
	\end{enumerate}
\end{proof}

Part (b) of Lemma \ref{lemma:xi} describes a possible completion step for part (a) that also modifies $\xi$ to be left-continuous.
In a nutshell, any decreasing $\xi' \colon \mathbb{Q} \to \mathcal{X}$ that satisfies $\xi'(\eta') \in X(\eta')$ for all $\eta' \in \mathbb{Q}$ admits a left-continuous version on $\R$, $\xi : \eta \mapsto \bigcap_{\eta' < \eta} \xi'(\eta') \in X(\eta)$, where the intersection is over all $\eta' \in \mathbb{Q}$, $\eta'<\eta$.
As $\eta$ increases, $\xi$ follows one of the totally ordered paths through the lattice.
In Figure \ref{fig:lattice} the direction of movement as $\eta$ increases is illustrated by arrows.

In order to prove the existence of a function $\xi$ (and thus $\hat{g}$) that solves the isotonic regression problem, we need that $\mathcal{X}$ is closed under union and intersection.
This property is essential for Lemma \ref{lemma:xi}.

We could also start with a set $\mathcal{X}$ of subsets of $\{z_1, \dots, z_n\}$ that are interpreted as the admissible superlevel sets of the function $g$ that is to be fitted.
If $\mathcal{X}$ is closed under union and intersection, then $\mathcal{X}$ induces a partial order on $\{z_1, \dots, z_n\}$ by Birkhoff's Representation Theorem; see for example \citet{Gurney2011}.
Consequently, the optimal function $\hat{g}$ always exists and is increasing.

Starting with $\mathcal{X}$, one could formulate constraints other than isotonicity on $g$ as long as they can be formulated in terms of restrictions on admissible superlevel sets.
Examples are unimodality or quasi-convexity.
Generally, there is no solution that is simultaneously  optimal with respect to all elementary loss functions; see Section \ref{sec:unimodal} for examples in the case of a unimodality constraint.

\begin{figure}
	\centering
	\includegraphics{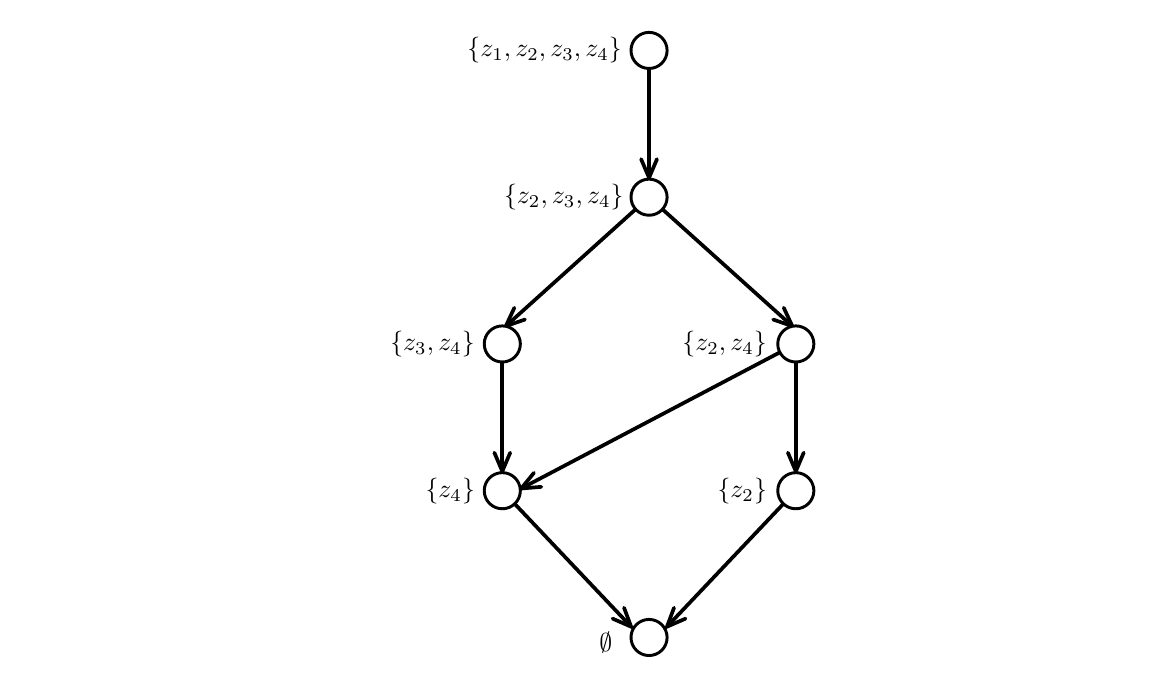}
	\caption{\textbf{Moving through $\mathcal{X}$.} The possible paths through $\mathcal{X}$ based on a specific partial order on $\mathcal{Z} = \{z_1,z_2,z_3,z_4\}$ are illustrated.
		The arrows indicate the direction of moving through the lattice $\mathcal{X}$ as $\eta$ increases
		\label{fig:lattice}}
\end{figure}

The following proposition generalizes Proposition \ref{prop:simplified_1} and is essential to provide min-max bounds on solutions to the isotonic regression problem.
We write $T_x^- = T_{P_x}^- = \inf T(P_x)$ and $T_x^+ = T_{P_x}^+ = \sup T(P_x)$.
\begin{proposition} \label{prop:generalized_1}
	Let $\eta \in \R$, $x \in X(\eta)$.
	Then, subject to $x' \in \mathcal{X}$,
	\begin{align*}
	\max_{x' \supsetneq x} T_{x' \setminus x}^-
	&\le \eta
	\le \min_{x' \subsetneq x} T_{x \setminus x'}^+,\\
	\max_{x' \supsetneq x, x' \notin X(\eta)} T_{x' \setminus x}^+
	&< \eta
	\le \min_{x' \subsetneq x, x' \notin X(\eta)} T_{x \setminus x'}^-.
	\end{align*}
\end{proposition}
\begin{proof}
	For all $x' \supsetneq x$, we have $v_{x' \setminus x}(\eta) \ge 0$.
	For all $x' \subsetneq x$, we have $v_{x \setminus x'}(\eta) \le 0$.
	If $x' \notin X(\eta)$, then both inequalities are strict.
	Corollary \ref{cor:TviaV} implies the result.
\end{proof}

As a generalization of Proposition \ref{prop:simplified_main1}, we obtain the following min-max bounds.
\begin{proposition} \label{prop:generalized_main}
	Let $z \in \mathcal{Z}$ and let $\hat{g}$ be a solution to the isotonic regression problem.
	Then, subject to $x, x' \in \mathcal{X}$,
	\[
	\min_{x' : z \notin x'} \max_{x \supsetneq x'} T_{x \setminus x'}^- \le \hat{g}(z)  \le \max_{x : z \in x} \min_{x' \subsetneq x} T_{x \setminus x'}^+.
	\]
\end{proposition}
\begin{proof}
	Applying the first set of bounds from Proposition \ref{prop:generalized_1} to the formula for $\hat{g}$ at \eqref{eq:hatg2}, we obtain
	\[\inf_{\eta : z \notin \xi(\eta)} \, \max_{x \supsetneq \xi(\eta)} T_{x \setminus \xi(\eta)}^- \le \hat{g}(z)  \le\max_{\eta : z \in \xi(\eta)} \, \min_{x' \subsetneq   \xi(\eta)} T_{\xi(\eta) \setminus x'}^+.\]
	The lower bound is bounded from below by 
	$\min_{x' : z \notin x'} \, \max_{x \supsetneq x'} T_{x \setminus x'}^-$, and the upper bound is bounded from above by $\max_{x : z \in x} \, \min_{x' \subsetneq x} T_{x \setminus x'}^+$.
\end{proof}

In the case of partial orders on the covariates, it is also possible to define minimal and maximal solutions.
Recall that, analogously to $I(\eta)$, we defined $X(\eta)$ as the set of superlevel sets $x \in \mathcal{X}$ minimizing $s_x(\eta)$ at \eqref{eq:generalmin}.
Now, let
\begin{align*}
X^-(\eta) &= \{x \in X(\eta) : \nexists\, x' \in X(\eta) \text{ such that } x' \subsetneq x \},
\\
X^+(\eta) &= \{x \in X(\eta) : \nexists\, x' \in X(\eta) \text{ such that } x' \supsetneq x\}
\end{align*}
denote the sets of minimal and maximal elements of $X(\eta)$, respectively.
In order to prove an analogous statement to Proposition \ref{prop:simplified_main2}, we need the following lemma on a modified max-min inequality.

\begin{lemma} \label{lemma:maxmin}
	Suppose that $T$ is of singleton type.
	Let $z \in \mathcal{Z}$ be such that $P(\{z\}\times\R) > 0$.
	Then, subject to $x, x' \in \mathcal{X}$,
	\[
	\max_{x : z \in x} \min_{x' \subsetneq x} T^+_{x \setminus x'} 
	\le \min_{x' : z \notin x'} \max_{x \supsetneq x'} T^-_{x \setminus x'}.
	\]
\end{lemma}

\begin{proof}
	Let $x'' \in \mathcal{X}$ such that $z \notin x''$, then 
	\begin{align*}
	\max_{x : z \in x}\, \min_{x' \subsetneq x} T^+_{x \setminus x'} &= \max_{x : z \in x} \min_{\substack{x' \subsetneq x\\P((x\setminus x') \times \R) > 0}} T^+_{x \setminus x'} = \max_{x : z \in x} \min_{\substack{x' \subsetneq x\\P((x\setminus x') \times \R) > 0}} T^-_{x \setminus x'}\\
	&\le  \max_{x : z \in x} T^-_{x \setminus (x \cap x'')}
	= \max_{x: z \in x} T^-_{(x \cup x'') \setminus x''}
	\le \max_{x: x \supsetneq x''} T^-_{x \setminus x''},
	\end{align*}
	where the last inequality holds because $x\cup x'' \in \mathcal{X}$ and if $z \in x$ then $x \cup x'' \supsetneq x''$.
\end{proof}

\begin{proposition}\label{prop:4.8}
	Let $z \in \mathcal{Z}$ be such that $P(\{z\} \times \R) > 0$, and let $\xi \colon \R \to \mathcal{X}$ be decreasing and left-continuous.
	\begin{enumerate}[label=(\alph*)]
		\item If $\xi(\eta) \in X^+(\eta)$ for all $\eta \in \R$, then, 
		subject to $x, x' \in \mathcal{X}$,
		\[
		\hat{g}(z)
		= \min_{x' : z \notin x'} \max_{x \supsetneq x'} 
		T_{x \setminus x'}^+
		= \max_{x : z \in x} \min_{x' \subsetneq x} T_{x \setminus x'}^+.
		\]
		
		\item
		If $\xi(\eta) \in X^-(\eta)$ for all $\eta \in \R$, then, 
		subject to $x, x' \in \mathcal{X}$,
		\[
		\hat{g}(z)
		= \min_{x' : z \notin x'} \max_{x \supsetneq x'} 
		T_{x \setminus x'}^-
		= \max_{x : z \in x} \min_{x' \subsetneq x} T_{x \setminus x'}^-.
		\]
	\end{enumerate}
\end{proposition}
\begin{proof}
	The proof follows using Lemma \ref{lemma:maxmin} and applying the same steps as in the proof of Proposition \ref{prop:generalized_main} to the second set of bounds in Proposition \ref{prop:generalized_1}.
\end{proof}

As in Section \ref{sec:total}, we denote the solution in part (a) of Proposition \ref{prop:4.8} by $g^+$ and the one in part (b) by $g^-$.
Combining Propositions \ref{prop:4.8} to \ref{prop:4.9} and Corollary \ref{cor:generated lattice}, gives a complete characterizations of all possible solutions to the isotonic regression problem for partial orders.

For the following results, it is not required that $g^-$, $g^+$ are the solutions from Proposition \ref{prop:4.8}.
Unless specified, they do not even need to satisfy $g^- \le g^+$ everywhere.
We define $\xi^- \colon \eta \mapsto \{z : {g}^-(z)\geq \eta\}$ and $\xi^+$ analogously.

\begin{proposition}\label{prop:4.6}
	Let $g^-$ and $g^+$ be two solutions to the isotonic regression problem such that $g^- \leq g^+$.
	Let $\hat{g}$ be isotonic, $g^- \leq \hat{g} \leq g^+$, and suppose that all superlevel sets of $\hat{g}$ lie in $\bigcup_{\eta \in \R} X(\eta)$. Then, $\hat{g}$ is a solution to the isotonic regression problem.
\end{proposition}
\begin{proof} 
	For $\eta \in \R$ define $\xi(\eta)= \{z : \hat{g}(z)\geq \eta\}$.
	The functions $\xi, \xi^-, \xi^+$ are decreasing, that is $\xi(\eta) \supseteq \xi(\eta')$ for $\eta \leq \eta'$, and left-continuous.
	For $\xi^-$, $\xi^+$ it holds that $\xi^-(\eta)$, $\xi^+(\eta) \in X(\eta)$.
	Since, for all $z\in \mathcal{Z}$, it holds that 
	\begin{align*}
	g^-(z) = \max \{ \eta : z \in \xi^-(\eta)\} &\leq
	g(z) = \max \{ \eta : z \in \xi(\eta)\} \\&\leq
	g^+(z) = \max \{ \eta : z \in \xi^+(\eta)\},
	\end{align*}
	we obtain $\xi^-(\eta) \subseteq \xi(\eta) \subseteq \xi^+(\eta)$ for all $\eta \in \R$.
	Lemma \ref{lem:4.1} (b) implies the result.
\end{proof}

The following corollary is an immediate consequence of Lemma \ref{lem:4.1} (c).
\begin{corollary}\label{cor:generated lattice}
	Let $g^-$ and $g^+$ be two solutions to the isotonic regression problem. 
	Then, the distributive lattice generated by $\xi^-$ and $\xi^+$ is a subset of $\bigcup_{\eta \in \R} X(\eta)$.
\end{corollary}

Having two solutions $g^-$ and $g^+$ allows us to find all solutions to the isotonic regression problem with superlevel sets that lie in the lattice generated by $\xi^-$ and $\xi^+$.
Examples include solutions that transition from $g^-$ to $g^+$ at a particular threshold $\eta$,
\[
\hat{g}(z) = \begin{cases} g^+(z), & z \in \xi^+(\eta), \\ g^-(z), & \text{otherwise}, \end{cases}
\]
or pointwise convex combinations of solutions with $\alpha \in (0, 1)$,
\[
\hat{g}(z) = \alpha g^-(z) + (1 - \alpha) g^+(z).
\]

In order to refine the lattice of minimizing upper sets from Corollary \ref{cor:generated lattice} with the purpose to characterize all solutions, we pose the question whether simple separation rules exist for the set difference of consecutive lattice elements.
These sets necessarily take the form of the intersection of a level set of $g^-$ and a level set of $g^+$, that is, sets of the form $\{z : g^-(z)= \eta^- \text{ and } g^+(z)= \eta^+ \}$.
These rules do exist as we show in Propositions \ref{prop:4.7} and \ref{prop:4.9}.
First, we introduce the notion of a separation.

\begin{definition}
	A \emph{separation} of a set $Z \in \mathcal{P}(\mathcal{Z})$ is a collection of sets $Z_1, \dots, Z_n \subseteq Z$ that are pairwise separated and satisfy $Z= \bigcup_{i=1}^n Z_i$.
	Two sets $Z_i$ and $Z_j$ are \emph{separated} with respect to $Z$ if for all $z' \in Z_i$ and $z'' \in Z_j$, there does not exist a finite sequence $(z_k)_{k=1, \dots,m}$, $z_k \in Z$, $z_1=z'$, $z_m=z''$ that for all $k=1,\dots,m-1$ satisfies $z_k \preceq z_{k+1}$ or $z_{k+1} \preceq z_k$.
\end{definition}

\begin{proposition}\label{prop:4.7} 
	Let $g^-$ and $g^+$ be two solutions to the isotonic regression problem, and let $\eta^-, \eta^+ \in \R$, $\eta^- < \eta^+$, be such that $Z=\{z : g^-(z)= \eta^- \text{ and } g^+(z)= \eta^+ \}$ is nonempty.
	Furthermore, let $Z_1, \dots, Z_n$ be a separation of $Z$, and let $x'=\xi^-(\eta^-) \cap \xi^+(\eta^+)$ and $x''=x' \setminus Z$.
	Then, $x'' \cup Z_k \in X(\eta)$ for all $\eta \in (\eta^-, \eta^+]$, $k = 1, \dots, n$.
\end{proposition}
\begin{proof}	
	Without loss of generality, we show the claim for $k = 1$.
	By Lemma \ref{lem:4.1} (c), we have $x' \in X(\eta^+)$ and $x'' = \xi^-(\eta^- + \epsilon_1) \cup \xi^+(\eta^+ + \epsilon_2)  \in X(\eta^-  + \epsilon_1)$ for some $\epsilon_1, \epsilon_2 > 0$.
	More precisely, we have $x', x'' \in X(\eta)$ for all $\eta \in (\eta^-, \eta^+]$ by Lemma \ref{lem:4.1} (b), since $\xi^-(\eta) \subseteq x'' \subseteq x' \subseteq \xi^+(\eta)$, $\eta \in (\eta^-, \eta^+]$.
	
	Let $x_1=x'' \cup Z_1$ and $x_2=x' \setminus Z_1$ both of which are upper sets in $\mathcal{X}$.
	Then $Z_1= x_1 \setminus x''$ but also $Z_1=x'\setminus x_2$.
	Therefore, $v_{Z_1}(\eta) \geq 0 \geq v_{Z_1}(\eta)$ for all $\eta \in (\eta^-, \eta^+]$ by Proposition \ref{prop:4.1}. Then the statement follows from Corollary \ref{cor:4.1}. 
\end{proof}

Proposition \ref{prop:4.7} allows us to find additional solutions to the isotonic regression problem with superlevel sets where separation elements have been added to known minimizing superlevel sets.
Using the variables defined in Proposition \ref{prop:4.7}, one example of a new solution is
\[
\hat{g}(z) = \begin{cases}
\eta, & z \in Z_1, \\
g^+(z), & z \in x'', \\
g^-(z), & \text{otherwise},
\end{cases}
\]
where $\eta \in (\eta^-, \eta^+]$.
Iterative application of Proposition \ref{prop:4.7} recovers all minimizing superlevel sets that can be obtained from the solutions in Proposition \ref{prop:4.8} via Corollary \ref{cor:generated lattice} and the information on the partially ordered set $\mathcal{Z}$.

Proposition \ref{prop:4.9} allows us to recover the remaining minimizing superlevel sets when the distribution $P$ of the random vector $(Z, Y)$ is fully known. In fact, this proposition is a generalization of Proposition \ref{prop:4.7} that determines whether a level set intersection of $g^-$ and $g^+$ can be split further by calculating values of the lower bound of the functional $T$.

\begin{proposition} \label{prop:4.9}
	Let $g^-$ and $g^+$ be two solutions to the isotonic regression problem, and let $\eta^-, \eta^+ \in \R$, $\eta^- < \eta^+$, be such that $Z=\{z : g^-(z)= \eta^- \text{ and } g^+(z)= \eta^+ \}$ is nonempty.
	Furthermore, let $x'=\xi^-(\eta^-) \cap \xi^+(\eta^+)$ and $x''=x' \setminus Z$.
	For $x \in \mathcal{X}$, $x' \supsetneq x \supsetneq x''$, we have $T_{x' \setminus x}^- \leq \eta^-$ if and only if $x \in X(\eta)$ for all $\eta \in (\eta^-, \eta^+]$.
\end{proposition}
\begin{proof}
	We have $x', x'' \in X(\eta)$ for all $\eta \in (\eta^-, \eta^+]$ as in the proof of Proposition \ref{prop:4.7}.
	Then, $v_{x' \setminus k} (\eta^+) \leq 0$ for all $k \in \mathcal{X}$, $k \subsetneq x'$, by Proposition \ref{prop:4.1}, and hence $T_{x' \setminus k}^+ \geq \eta^+$ by Corollary \ref{cor:TviaV}.
	Analogously, $v_{k \setminus x''}(\eta) \geq 0$ for all $k \in \mathcal{X}$, $k\supsetneq x''$, $\eta \in (\eta^-, \eta^+]$, leading to $T_{k \setminus x''}^- \leq \eta^-$.
	
	For the first part of the statement, let $x \in \mathcal{X}$, $x' \supsetneq x \supsetneq x''$, be such that $T_{x'\setminus x}^- \leq \eta^-$.
	We show that $x \in X(\eta)$ for all $\eta \in (\eta^-,\eta^+]$ using Proposition \ref{prop:4.1}.
	We have $T_{x' \setminus k}^+ \leq \max \{ T_{x' \setminus x}^-, T_{x \setminus k}^+\}$ for all $k \subsetneq x$ by Lemma \ref{lemma:functional_bounds}.
	Since $T_{x' \setminus x}^- \leq \eta^-$ by assumption and as just shown $T_{x' \setminus k}^+ \geq \eta^+$, we obtain $T_{x \setminus k}^+ \geq \eta^+$.
	By Corollary \ref{cor:TviaV}, $v_{x \setminus k}(\eta) \leq 0$ for all $k \subsetneq x$, $\eta \leq \eta^+$, that is, the first inequality in Proposition \ref{prop:4.1} holds for all $\eta \in (\eta^-,\eta^+]$.
	Similarly, $T_{k \setminus x''}^- \geq \min \{T_{k \setminus x}^-, T_{x \setminus x''}^+\}$ for all $k \supsetneq x$.
	Since $T_{k \setminus x''}^- \leq \eta^-$ and $T_{x \setminus x''}^+ \geq \eta^+$, we obtain $T_{k \setminus x}^- \leq \eta^-$.
	Therefore, $v_{k \setminus x}(\eta) \geq 0$, for all $\eta > \eta^-$, $k \supsetneq x$, that is, the second inequality in Proposition \ref{prop:4.1} holds for all $\eta \in (\eta^-,\eta^+]$.
	
	To prove the converse, note that $x \in X(\eta)$ for all $\eta \in (\eta^-, \eta^+]$ implies that $v_{k \setminus x}(\eta) \geq 0$ for all $\eta \in (\eta^-, \eta^+]$, $k \supsetneq x$.
	Hence, in particular, $v_{x' \setminus x}(\eta) \geq 0$ and $T_{x' \setminus x}^- \leq \eta$ for all $\eta \in (\eta^-, \eta^+]$, and, therefore $T_{x' \setminus x}^- \leq \eta^-$.
\end{proof}

\subsection{Partitioning the covariate set} \label{sec:partition}
In Section \ref{sec:PAV}, we discussed how the PAV algorithm creates a partition of $\mathcal{Z}$, and that it leads to a solution $\hat{g}$ of the isotonic regression problem in the context of total orders.
In this section, we show how a solution to the isotonic regression problem leads to a corresponding partition $\mathcal{Q}$ of $\mathcal{Z}$, such that the solution satisfies
\begin{align*}
\hat{g}(z) \in T(P_Q), \quad \text{for all } Q \in \mathcal{Q},\,\, z \in Q,
\end{align*}
and the solution is constant on every element of the partition.
Let $T$ be a functional of singleton type, and $\hat{g}$ be a solution to the isotonic regression problem.
Subject to $x, x', k, k' \in \mathcal{X}$, the combination of Proposition \ref{prop:generalized_main} and Lemma \ref{lemma:maxmin} yields
\begin{align}
\hat{g}(z) 
&= \max_{x : z \in x} 
\min_{x' \subsetneq x} 
T_{x\setminus x'}^+\label{eq:1}
\\&= \min_{k' : z \notin k'} 
\max_{k \supsetneq k'} 
T_{k\setminus k'}^-.\label{eq:2}
\end{align}
for all $z \in \mathcal{Z}$ with $P(\{z\} \times \R) > 0$.
We call $(x, x')$ a max-min pair for $z$ if $z \in x$, $x' \subsetneq x$, and $\hat{g}(z) = T_{x \setminus x'}^+$, and we call $(k', k)$ a min-max pair for $z$ if $z \notin k'$, $k \supsetneq k'$, and $\hat{g}(z) = T_{k \setminus k'}^-$.
For a pair $x, x' \in \mathcal{X}$ such that $T_{x \setminus x'}^- = T_{x \setminus x'}^+$, we also use the notation $T_{x \setminus x'}^\pm$.
Note that for a functional $T$ of singleton type, we have $T(P_{x \setminus x'}) = \{ T_{x \setminus x'}^\pm \}$ if $P((x \setminus x') \times \R) > 0$.
The following lemma provides the necessary tools to construct the partition $\mathcal{Q}$.
\begin{lemma}\label{lem:4.3}
	Let $T$ be a functional of singleton type, and $\hat{g}$ be a solution to the isotonic regression problem.
	Furthermore, let $z \in \mathcal{Z}$ such that $P(\{z\} \times \R) > 0$, and let $(x_1, x_1'), (x_2, x_2')$ be max-min pairs for $z$, and $(k_1', k_1), (k_2', k_2)$ be min-max pairs for $z$.
	Then the following statements hold:
	\begin{enumerate}[label=(\alph*)]
		\item We have that $\hat{g}(z) = T_{x_1 \setminus k_1'}^\pm = T_{(x_1 \cup x_2) \setminus k_1'}^\pm = T_{x_1 \setminus (k_1' \cap k_2')}^\pm$.
		\item If $x, k' \in \mathcal{X}$ such that $z \in x$, $z \notin k'$, and $\hat{g}(z) = T_{x \setminus k'}^\pm$, then $(x, x \cap k')$ is a max-min pair for $z$ and $(k', k' \cup x)$ is a min-max pair for $z$.
		\item If $\tilde{z} \in x_1 \setminus k_1'$, then $(x_1, x_1')$ is a max-min pair for $\tilde{z}$ and $(k_1', k_1)$ is a min-max pair for $\tilde{z}$.
	\end{enumerate}
\end{lemma}
\begin{proof}
	We repeatedly use the inequalities $\hat{g}(z) = T_{x_1 \setminus x_1'}^+ = \min_{x' \in \mathcal{X}} T_{x_1 \setminus x'}^+ \le T_{x_1 \setminus k'}^+$ and $\hat{g}(z) = T_{k_1 \setminus k_1'}^- = \max_{k \in \mathcal{X}} T_{k \setminus k_1'}^- \ge T_{x \setminus k_1'}^-$ for all $x, k' \in \mathcal{X}$, where the second equality holds because $T_P^+ = \infty$ and $T_P^- = -\infty$ for null measures $P$.
	Furthermore, by assumption, $T(P_{x \setminus k'})$ is a singleton if $P((x \setminus k') \times \R) > 0$, and therefore $T(P_{x \setminus k'})$ is a singleton if $z \in x$ and $z \notin k'$.
	\begin{enumerate}[label=(\alph*)]
		\item Clearly, $z \in x_1$, $z \in x_2$, $z \notin k_1'$, and $z \notin k_2'$.
		Hence, $\hat{g}(z) \le T_{x_1 \setminus k_1'}^\pm \le \hat{g}(z)$ implies the first statement.
		Furthermore, $\hat{g}(z) \le T_{x_2 \setminus (x_1 \cup k_1')}^+ = T_{(x_2 \setminus x_1) \setminus k_1'}^+$, and hence $\hat{g}(z) = \min\{ T_{x_1 \setminus k_1'}^-, T_{(x_2 \setminus x_1) \setminus k_1'}^+ \} \le T_{(x_1 \cup x_2) \setminus k_1'}^\pm \le \hat{g}(z)$ confirms the second statement using Lemma \ref{lemma:functional_bounds}.
		Similarly, for the third statement, $\hat{g}(z) \le T_{x_1 \setminus (k_1' \cap k_2')}^\pm \le \max\{ T_{x_1 \setminus k_1'}^+, T_{(x_1 \cap k_1') \setminus k_2'}^- \} = \hat{g}(z)$.
		\item The statement follows immediately from $T_{x \setminus k'}^- = T_{x \setminus k'}^+$, $(x \cup k') \setminus k' = x \setminus k' = x \setminus (x \cap k')$, and the definition of max-min and min-max pairs.
		\item Let $(x_{\tilde{z}}, x_{\tilde{z}}')$ be a max-min pair for $\tilde{z}$ and $(k_{\tilde{z}}', k_{\tilde{z}})$ be a min-max pair for $\tilde{z}$.
		Then the statement follows from $\hat{g}(z) \le T_{x_1 \setminus k_{\tilde{z}}'}^\pm \le \hat{g}(\tilde{z}) \le T_{x_{\tilde{z}} \setminus k_1'}^\pm \le \hat{g}(z)$.
	\end{enumerate}
\end{proof}
\begin{proposition}
	Let $T$ be a functional of singleton type.
	Then there exists a partition $\mathcal{Q}$ of $\mathcal{Z}$ such that $\hat{g}$ is constant on every element of the partition almost everywhere and $\hat{g}(z) \in T(P_Q)$ for all $Q \in \mathcal{Q}$, $z \in Q$ such that $P(\{z\} \times \R) > 0$.
\end{proposition}
\begin{proof}
	Let $\bar{x}_z$ denote the union of the first components of all max-min pairs for $z \in \mathcal{Z}$, and let $\bar{k}_z'$ denote the intersection of the first components of all min-max pairs for $z \in \mathcal{Z}$.
	By Lemma \ref{lem:4.3} (a), we have $\hat{g}(z) = T_{{\bar{x}_z} \setminus \bar{k}_z'}^\pm$.
	We now show that the collection $\mathcal{Q}$ of sets $Q_z = \bar{x}_z \setminus \bar{k}_z'$ is a partition of $\mathcal{Z}$.
	First, we have $\bigcup_{z \in \mathcal{Z}} Q_z = \mathcal{Z}$, since $z \in \bar{x}_z$ and $z \notin \bar{k}_z'$ for all $z \in \mathcal{Z}$.
	Second, by Lemma \ref{lem:4.3} (b), we have that $(\bar{x}_z, \bar{x}_z \cap \bar{k}_z')$ is a max-min pair for $z$ and $(\bar{k}_z', \bar{k}_z' \cup \bar{x}_z)$ is a min-max pair for $z$.
	Then, by Lemma \ref{lem:4.3} (c), we have $\bar{x}_z \subset \bar{x}_{\tilde{z}}$ and $\bar{k}_z' \supset \bar{k}_{\tilde{z}}'$ for all $\tilde{z} \in Q_z$, i.e., $Q_z \subset Q_{\tilde{z}}$ and in particular $z \in Q_{\tilde{z}}$.
	Swapping the roles of $z$ and $\tilde{z}$ gives $Q_{\tilde{z}} \subset Q_z$.
	Therefore, $Q_z = Q_{\tilde{z}}$ for all $z \in \mathcal{Z}, \tilde{z} \in Q_z$.
\end{proof}

When $T$ is a functional of interval type, we therefore obtain a partition for every fixed convex combination of its lower bound $T^-$ and its upper bound $T^+$.

\section{Unimodal Regression} \label{sec:unimodal}
It is astonishing that in isotonic regression, solutions are simultaneously optimal for all loss functions in the class $\mathcal{S}$ which exhausts all consistent loss functions for the functional $T$ in many relevant examples.
One might wonder whether this is still fulfilled for slightly adapted shape constraints.
Unimodality is a shape constraint closely related to isotonicity.
One estimation procedure is to take a mode between two consecutive observations and then split the data set in two.
On the subset preceding the mode an isotonic regression is performed and on the data following the mode an antitonic regression is performed.
This procedure is then repeated for any possible choice of mode, as illustrated in Figure \ref{fig:partition}.
Finally the optimal function is chosen by selecting the one with minimal loss.
The reason for the mode to be chosen outside of $\{z_1, \dots, z_n\}$ is to avoid ambiguity.
If the mode is fixed on observation $z_i$, $1 < i < n$, then the isotonic regression on $\{ z_1, \dots, z_i \}$ and the antitonic regression on $\{ z_i, \dots, z_n \}$ might yield two different values for $\hat{g}(z_i)$.

\begin{figure}
	\centering
	\includegraphics{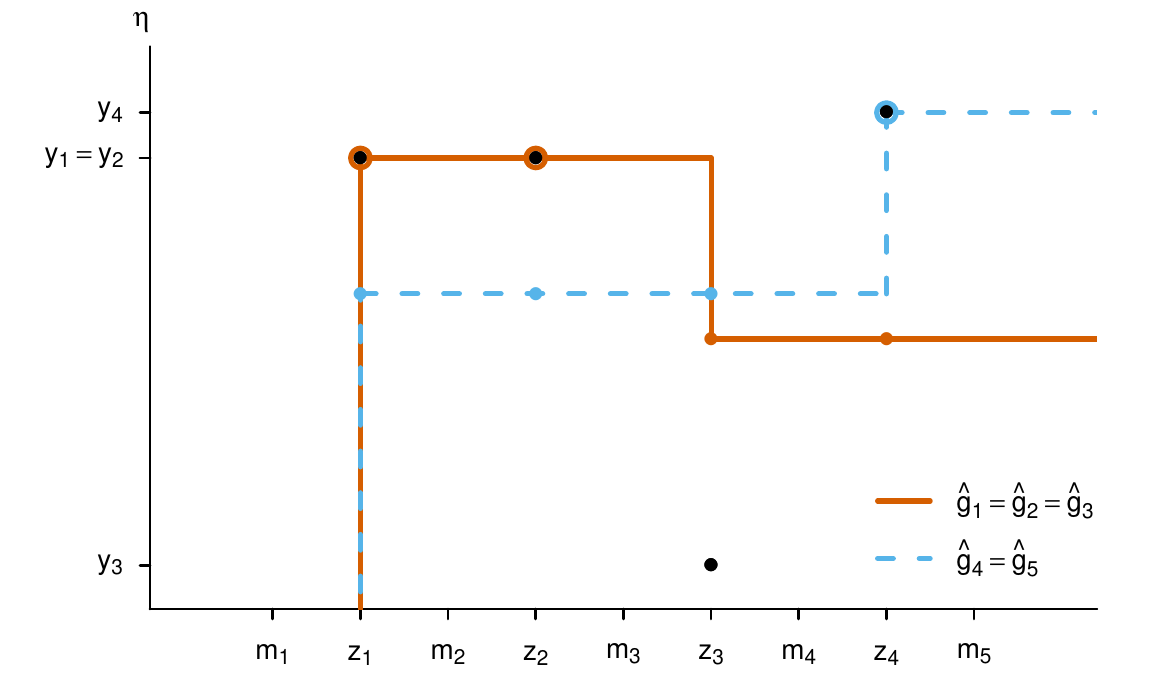}
	\caption{\textbf{Possible Modes.} For a sample of 4 data points, the five possible choices $m_1, \dots, m_5$ for the mode, and the corresponding subdivision into isotonic and antitonic part for the functions $\hat{g}_1, \dots, \hat{g}_5$ are marked\label{fig:partition}}
\end{figure}

Fixing mode $m_i$ and applying our method to $\{z_1, \dots, z_{i-1}\}$ and $\{z_i,\dots, z_n\}$  with isotonicity and antitonicity, respectively, as shape constraints yields a function $\hat{g}_i \colon \{z_1, \dots, z_n\} \to \R$ that is optimal for any consistent loss function for functional $T$.
The question arises whether there is one mode $m_i$ such that the corresponding $\hat{g}_i$ dominates all other functions $\hat{g}_j$, $j\neq i$.
It turns out that this is generally not the case.

To give an example, we consider four observations $(z_1, y_1), \dots, (z_4, y_4)$ with $z_1 < \dots < z_4$ and $(y_1, \dots y_4) = (9, 9, 0, 10)$, and let $P$ denote the corresponding empirical distribution.
We choose the expectation functional as the regression target, and consider modes $m_1, \dotsc, m_5$ with $m_1 < z_1 < m_2 < z_2 < \dots  < z_4 < m_5$.
For modes $m_1$ and $m_3$, the unimodal approach yields the partitions $\mathcal{Q}_{m_1}=\mathcal{Q}_{m_3}=\{\{z_1,z_2\},\{z_3, z_4\}\}$ using the PAV algorithm, and for mode $m_2$, we obtain the partition $\mathcal{Q}_{m_2}=\{\{z_1\}, \{z_2\}, \{z_3,z_4\}\}$.
But for this specific data example all three modes yield the same function i.e., $\hat{g}_1 = \hat{g}_2= \hat{g}_3$.
For modes $m_4$ and $m_5$, we obtain the partitions $\mathcal{Q}_{m_4}=\mathcal{Q}_{m_5}=\{\{z_1,z_2,z_3\},\{z_4\}\}$ and therefore $\hat{g}_4=\hat{g}_5$.
The functions $\hat{g_1}, \dots, \hat{g}_5$ are illustrated in Figure \ref{fig:counterexample}.

\begin{figure}
	\centering
	\includegraphics{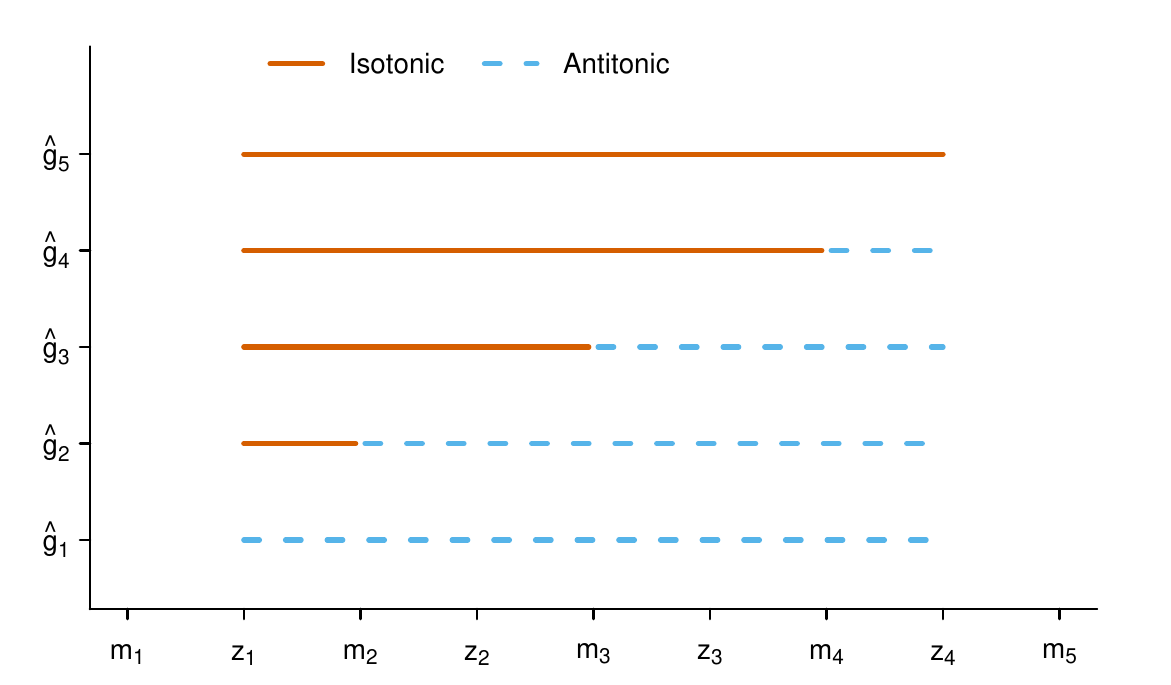}
	\caption{\textbf{Counterexample.} For the specific sample of 4 data points (black), the five possible choices $m_1, \dots, m_5$ for the mode, and the resulting functions $\hat{g}_1, \dots, \hat{g}_5$ are illustrated\label{fig:counterexample}}
\end{figure}

Solution $\hat{g}_i$, $1\leq i \leq 5$, dominates all other $\hat{g}_j$, $j\neq i$, if
\begin{align*}
\E_{P} S_\eta (\hat{g}_i(Z),Y) \leq \E_{P} S_\eta (\hat{g}_j(Z),Y) 
\quad \text{for all } \eta \in \R, j \neq i.
\end{align*}
It can be seen that this condition is not fulfilled by plotting the expected elementary scores for $\hat{g}_1, \dots, \hat{g}_5$; see Figure \ref{fig:murphy}.
This visual method of comparing forecasts is called a Murphy diagram and was introduced by \citet{Ehm2016}.

Hence, in unimodal regression there is not necessarily a solution $\hat{g}_i$ that simultaneously minimizes all consistent loss functions for a functional $T$.
This agrees with our findings in Section \ref{sec:partial} because the set $\mathcal{X}$ is not closed under union and  intersection.
Indeed, it holds that $\{z_1\}, \{z_4\} \in \mathcal{X}$ but $\{z_1,z_4\} \notin \mathcal{X}$.
Therefore, the existence of a decreasing function $\xi \colon \R \to \mathcal{X}$ is not guaranteed.

\begin{figure}
	\centering
	\includegraphics{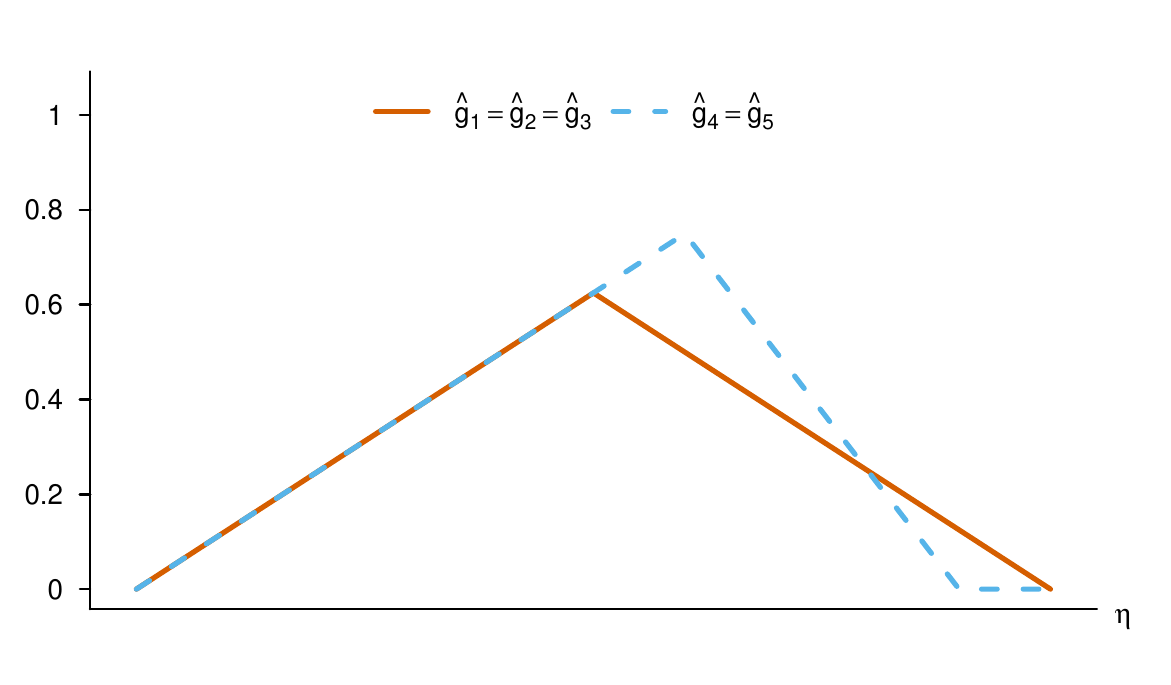}
	\caption{\textbf{Murphy Diagram.} The Murphy diagram comparing the expected elementary scores of $\hat{g}_1, \dots, \hat{g}_5$ for $\eta \in \R$ given realizations $(y_1, \dots, y_4) = (9, 9, 0, 10)$\label{fig:murphy}}
\end{figure}

\pagebreak
{\small \textbf{Acknowledgements}
	We would like to thank Tilmann Gneiting, Alexandre M\"oschung and Lutz D\"umbgen for inspiring discussions and valuable comments.
	Anja M\"uhlemann and Johanna F.~Ziegel gratefully acknowledge financial support from the Swiss National Science Foundation.
}

\bibliographystyle{myims2}
\bibliography{biblio_pav}

\end{document}